\documentclass[twoside, 11pt]{article}

\newcommand{\Addresses}{{% additional braces for segregating \footnotesize
  \bigskip
  \footnotesize

  \textsc{Department of Mathematics,
University of Michigan-Ann Arbor,
530 Church Street, 
Ann Arbor, MI 48109-1043, USA}\par\nopagebreak
  \textit{E-mail address}: \texttt{elduque@umich.edu}

}}

\usepackage{amssymb}                                                                 					% Creates AMS Symbols 
\usepackage{amsthm}  
\usepackage{amsrefs}
\usepackage{amsmath}
\usepackage{mathrsfs}  								    					% Sheafy font \mathscr{}
\usepackage{verbatim}
 
\usepackage[hmargin=.75 in, vmargin=.75 in]{geometry}                	   					% See geometry.pdf to learn the layout options. There are lots.
\usepackage{graphicx}								  					 % Allows one to use jpeg, eps, and other images
\usepackage{float}									    					 % Allows one to use H when placing figures
\usepackage{epstopdf}
\usepackage{pgfplots}
\pgfplotsset{compat=1.12}	
\usepackage{subcaption}

\usepackage{tikz}
\usetikzlibrary{decorations.markings}									    					% Allows us to create diagrams
\usetikzlibrary{matrix}													% Calls useful TikZ libraries
\usetikzlibrary{shapes}													% Calls useful TikZ libraries

\usepackage[cmtip,arrow]{xy}												% Nice for creating commutative diagrams
\usepackage{pb-diagram,pb-xy}											% Nice for creating commutative diagrams
\usepackage{tikz-cd}									    					% Nice for creating commutative diagrams

\usepackage{enumerate}								     					% Helps making lists and outlines
\usepackage{thmtools}
\usepackage{thm-restate}
\usepackage{hyperref}
\usepackage[capitalise]{cleveref}
\usepackage{mathtools}
\usepackage{bbm}									     					% Creates nice blackboard bold 1
\usepackage{framed}	
\usepackage{eufrak}														%Fraktur Font
\newtheorem{theorem}{Theorem}

\newtheorem{cor}{Corollary}
\newtheorem{lem}{Lemma}
\newtheorem{prop}{Proposition}
\theoremstyle{definition}
\newtheorem{defn}{Definition}

\newtheorem{notn}{Notation}

\newtheorem{exm}{Example}
\newtheorem{rem}{Remark}

\newcommand{\Hom}{\text{Hom}}
\newcommand{\Id}{\text{Id}}
\newcommand{\End}{\text{End}}
\newcommand{\Ext}{\text{Ext}}
\newcommand{\rank}{\operatorname{rank}}					%l length
\newcommand{\into}{\hookrightarrow}						% Injection
\newcommand{\cA}{\mathcal{A}}
\newcommand{\cB}{\mathcal{B}}

\newcommand{\cH}{\mathcal{H}}

\newcommand{\cL}{\mathcal{L}}
\newcommand{\cM}{\mathcal{M}}

\newcommand{\cS}{\mathcal{S}}

\newcommand{\cU}{\mathcal{U}} 
\newcommand{\cV}{\mathcal{V}}

\newcommand{\B}{\mathbb{B}}
\newcommand{\C}{\mathbb{C}}

\newcommand{\F}{\mathbb{F}}

\newcommand{\K}{\mathbb{K}}

\renewcommand{\P}{\mathbb{P}}
\newcommand{\Q}{\mathbb{Q}}
\newcommand{\R}{\mathbb{R}}

\newcommand{\V}{\mathbb{V}}
\newcommand{\W}{\mathbb{W}}

\newcommand{\Z}{\mathbb{Z}}

\newcommand{\vepsi}{\varepsilon}									% Variable Epsilon
\newcommand{\Ht}[2]{H_{#1}^{\varepsilon,\rho}(#2,\F[t^{\pm 1}])}
\newcommand{\Htc}[2]{H_{#1}^{\varepsilon,\rho}(#2,\C[t^{\pm 1}])}	%Twisted Alexander Modules
\newcommand{\Ft}{\F[t^{\pm 1}]}				%Laurent Polynomial Ring
\newcommand{\f}[5]{
\begin{array}{rcl}
#1:#2 & \longrightarrow & #3 \\
#4 & \longmapsto & #5  \\
\end{array}
}						

\DeclareMathOperator{\GL}{\operatorname{GL}}							% General Linear Group
\DeclareMathOperator{\deti}{\det\nolimits_{\vepsi,\rho}}
\DeclareMathOperator{\Ima}{Im}

\DeclareMathOperator{\id}{Id}										% Stabilizer
\DeclareMathOperator{\Aut}{Aut}										% Automorphism Group 
\DeclareMathOperator{\Gal{Gal}}										% Galois Group

\begin{document}
\title{Twisted Alexander Modules of Hyperplane Arrangement Complements
}


\author{Eva Elduque %etc.
}
\date{\vspace{-5ex}}

\maketitle

\begin{abstract}
We study torsion properties of the twisted Alexander modules of the affine complement $M$ of a complex essential hyperplane arrangement, as well as those of punctured stratified tubular neighborhoods of complex essential hyperplane arrangements. We investigate divisibility properties between the twisted Alexander polynomials of the two spaces, compute the (first) twisted Alexander polynomial of a punctured stratified tubular neighborhood of an essential line arrangement, and study the possible roots of the twisted Alexander polynomials of both the complement and the punctured stratified tubular neighborhood of an essential hyperplane arrangement in higher dimensions. We apply our results to distinguish non-homeomorphic homotopy equivalent arrangement complements. We also relate the twisted Alexander polynomials of $M$ with the corresponding twisted homology jump loci.
\end{abstract}

\bigskip



\section{Introduction}

The twisted Alexander polynomial was first used to study plane algebraic curves by Cogolludo and Florens in \cite{cogo}. In their paper, they refine Libgober's divisibility results regarding the classical Alexander polynomial (\cite{lib3}), and use the twisted Alexander polynomials to distinguish Zariski pairs (pairs of plane curves with homeomorphic tubular neighborhoods but non-homeomorphic complements) that the classical Alexander polynomial cannot distinguish.

Cohen and Suciu study the multivariable twisted Alexander polynomials of the boundary manifold of a line arrangement in \cite{suciu}, and use the non-twisted version to obtain a complete description of the first characteristic variety of the fundamental group of the boundary manifold. Hironaka \cite{hironaka} and Florens-Guerville-Marco \cite{mam} have studied relationships between the topology of a line arrangement complement and that of the boundary manifold of such an arrangement.

In \cite{maxtommy}, Maxim and Wong investigated torsion properties for the twisted Alexander modules of the affine complements of complex hypersurfaces in general position at infinity. They did so by using the link (complement) at infinity, which fibers over a circle and ``dominates'' the hypersurface complement in the sense that the hypersurface complement can be obtained from it by adding cells of dimension greater or equal than the middle dimension. They were also able to describe a polynomial such that the roots of the (one-variable) twisted Alexander polynomials of the hypersurface complements were roots of it. This polynomial came from studying the twisted Alexander modules of the link at infinity.

Kohno and Pajitnov showed in \cite{morse} that complex essential hyperplane arrangements also had a similar structure. Hyperplane arrangements are not necessarily in general position at infinity, but there is a different space $V$ that plays a similar role as the one the link at infinity plays in the case of hypersurfaces in general position at infinity studied by Maxim and Wong. This space  is the boundary of a certain neighborhood of the arrangement, it fibers over a circle, and ``dominates'' the arrangement complement in the sense that the arrangement complement can be obtained from it by adding cells of the middle dimension.

In this paper, we follow Maxim and Wong's approach of using a ``dominating'' space to study the torsion properties for the twisted Alexander modules in the case of complex essential hyperplane arrangement complements, using the structure proved by Kohno and Pajitnov. We investigate divisibility properties between the twisted Alexander polynomials of arrangement complements and those of punctured tubular neighborhoods of arrangements, compute the (first) twisted Alexander polynomial of a punctured stratified tubular neighborhood of an essential line arrangement, and study the possible roots of the twisted Alexander polynomials of both the complement and the punctured stratified tubular neighborhood of an essential hyperplane arrangement in higher dimensions. To be able to do so, we define and study the topology of a punctured stratified tubular neighborhood $W^*$ of the arrangement. In particular, we prove that $W^*$ shares useful ``dominating'' properties with Kohno and Pajitnov's $V$, but, unlike $V$, it is well suited for our computations due the stratified nature of its definition. 

In the last section we give two applications of our results. The first application (\cref{exmfalk}) uses twisted Alexander polynomials to distinguish two non-homeomorphic homotopy equivalent line arrangement complements. The result proved in \cref{exmfalk} was already known: it was first proved by Jiang and Yau in \cite{yau} as a corollary of their powerful result that states that homeomorphic complex projective line arrangement complements have isomorphic intersection posets, and later reproved by Cohen and Suciu in \cite{suciu} using multivariable Alexander polynomials. The interesting part about our proof given in \cref{exmfalk} is that it does not rely on the heavy machinery of Jiang and Yau and it uses an a priori easier invariant than Cohen and Suciu, since multivariable Alexander polynomials are in principle harder to compute than univariable (twisted) ones due to the fact that they do not live in a PID. The second application relates the zeros of twisted Alexander polynomials to the twisted homology jump loci of rank one $\C$-local systems.

\subsection{Setup}
\label{ssdef}

Let $H_j$ be a complex hyperplane in $\C^n$ given by the zero locus of an affine linear map $\xi_j:\C^n\rightarrow \C$, where $j=1,\ldots, m$.

\begin{defn}
The hyperplane arrangement $\mathcal{A}=\{H_1,\ldots,H_m\}$ is called \textit{essential} if the maximal codimension of a non-empty intersection of a subfamily of $\mathcal{A}$ is $n$.
\end{defn}

Let $\{H_1,\ldots,H_m\}$ be an essential hyperplane arrangement, let
$$
H=\bigcup_{j=1}^m H_j
$$
be the union of the hyperplanes, and let
$$
M=\C^n\backslash H
$$
be the complement in $\C^n$.

\begin{rem}
Every hyperplane arrangement complement is homotopy equivalent to the complement of an essential one in an affine space of less or equal dimension (\cite[Proposition 6.1]{morse}), so we do not lose information by restricting ourselves to the study of essential hyperplane arrangements.
\end{rem}

Now, we need to identify and name certain loops in $\pi_1(M)$ that will be used throughout the paper. For a complete algorithm describing a presentation of $\pi_1(M)$, we refer the reader to \cite{arv}. It is a well-known fact that $\pi_1(M)$ is generated by a choice of meridians $a_j$ around each hyperplane $H_j$, for $j=1,\ldots,m$. These meridians $a_1,\ldots,a_m$ have a canonical (positive) orientation induced by the complex structure.

For the rest of \cref{ssdef}, we will deal with the case where $\mathcal{A}$ is a line arrangement (that is, $n=2$). Let $P_1,\ldots, P_s$ be the singular points of $H$, and let $d_k$ be the number of lines in $\cA$ passing through $P_k$.

\begin{defn}
We denote by $M_k$ the local complement
$$
M_k:=M\cap \B^4_k
$$
where $\B^4_k$ is a small 4-ball in $\C^2$ centered at the point $P_k$, for $k=1,\ldots,s$.
\end{defn}

Note that $M_k$ is homotopy equivalent to $M\cap S^3_k$, where $S^3_k$ is the boundary of $\overline{\B^4_k}$. In fact, $M\cap S^3_k=S^3_k\backslash L_k$,
where $L_k$ is a Hopf link with $d_k$ components. Also $M_k$ is naturally homeomorphic to a central line arrangement complement $U_k\subset\C^2$ consisting on $d_k$ distinct lines passing through the origin.

\begin{defn}
We denote by $\beta_k$ the loop in $\pi_1(M_{k})$ corresponding via the homeomorphism described above to a meridian about the line at infinity with \textbf{negative} orientation in $U_{k}$.
\label{defbeta}
\end{defn}

\begin{rem}
Two meridians about the same line with the same orientation are not necessarily the same elements in $\pi_1(M)$, but they are conjugate to one another. The above definition of $\beta_k$ for all $k=1,\ldots, s$ is well defined only up to conjugation in $\pi_1(M_{k})$, but this will suffice for our purposes. Abusing notation and disregarding base points, we will look at the $\beta_k$'s inside of $\pi_1(M)$ via the maps $\pi_1(M_{k})\longrightarrow\pi_1(M)$ induced by inclusion.
\end{rem}

\begin{rem}
$\beta_k$ can be taken to be the composition of $d_k$ loops $\gamma_1\cdot\ldots\cdot\gamma_{d_k}$, where each one of these loops is a certain positively oriented meridian about each of the $d_k$ lines in $\cA$ going through $P_k$. If the reader wishes to know what line of $\cA$ corresponds to each of the $\gamma$'s in a given example, they can do so using Arvola's presentation for $\pi_1(M)$ (\cite{arv}). For the purposes of this paper, we just need to know that a presentation for the fundamental group of $M_k$ (and $S^3_k\backslash L_k$) is given by 
\begin{equation}
\langle \beta_k,\gamma_1,\ldots,\gamma_{d_k-1}\mid [\beta_k,\gamma_l]\text{ \ for }l=1,\ldots,d_k-1\rangle
\label{eqnLink}
\end{equation}
(see \cite[Lemma 2.7]{maxtommy}).
\label{remHopf}
\end{rem}

\subsection{General construction of Alexander modules and polynomials}
\label{ssgen}

Let $\F$ be a field, and let $\V$ be a finite dimensional $\F$-vector space. Let $X$ be a path-connected finite CW complex, let $\rho:\pi_1(X)\rightarrow \GL(\V)$ be a linear representation, and let $\vepsi:\pi_1(X)\rightarrow \Z$ be a group homomorphism. Together, $\rho$ and $\vepsi$ define the homological twisted Alexander modules $\Ht{i}{X}$, as in \cite[Section 2.1]{maxtommy}.

\begin{defn}
The $i$-th (homological) twisted Alexander module $\Ht{i}{X}$ of $(X,\vepsi,\rho)$ is the $i$-th homology of the complex of $\Ft$-modules
$$
C^{\vepsi,\rho}_*(X,\Ft):=(\Ft\otimes_\F \V)\otimes_{\F[\pi_1(X)]} C_*(\widetilde X,\F).
$$
Here, $C_*(\widetilde X,\F)$ is the cellular homology complex of the universal cover $\widetilde X$ of $X$, seen as a free left $\F[\pi_1(X)]$-module via the action given by deck transformations. We regard $\Ft\otimes_\F \V$ as a right $\F[\pi_1(X)]$-module, with the right action given by
$$
(p(t)\otimes v)\cdot \alpha=(p(t)\cdot t^{\vepsi(\alpha)})\otimes (v \cdot \rho(\alpha))
$$
for every $p(t)\in \Ft$, $v\in\V$ and $\alpha\in \pi_1(X)$, where $v$ is regarded as a row vector and $\rho(\alpha)$ as a square matrix.
\label{defntwistmod}
\end{defn}

Together, $\vepsi$ and $\rho$ define a tensor representation
$$
\f{\vepsi\otimes\rho}{\pi_1(X)}{\Aut_{\Ft}(\Ft\otimes_\F \V)}{\alpha}{p(t)\otimes v\mapsto (p(t)\cdot t^{\vepsi(\alpha)})\otimes (v\cdot \rho(\alpha))}
$$
which gives rise to a local system of $\Ft$ modules $\mathcal{L}_{\vepsi,\rho}$.

\begin{rem}
There is an $\Ft$-module isomorphism $\Ht{i}{X}\cong H_i(X,\mathcal{L}_{\vepsi,\rho})$ (see \cite[Section 4.4]{maxtommy}). We will use both the chain complex definition and properties of homology with local systems when it is most convenient.
\end{rem}

Since $X$ is a finite CW-complex and $\V$ is finite dimensional over $\F$, $C_*(\widetilde X, \F)$ is a complex of finitely generated free left $\F[\pi_1(X)]$-modules. Thus, the twisted (homological) Alexander modules are finitely generated $\Ft$-modules over the principal ideal domain $\Ft$, and therefore have a direct sum decomposition into cyclic modules.

\begin{defn}
The $i$-th (homological) twisted Alexander polynomial of $(X,\vepsi,\rho)$ is defined as the order of the torsion part of the $i$-th twisted Alexander module $\Ht{i}{X}$. We denote this polynomial by $\Delta_i^{\vepsi,\rho}(X)$, and it is an element in $\Ft$ that is well defined up to multiplication by a unit of $\Ft$. If $\Ht{i}{X}$ is free, then $\Delta_i^{\vepsi,\rho}(X)=1$ by convention.

Equivalently, $\Delta_i^{\vepsi,\rho}(X)$ can be defined as a generator of the first non-zero Fitting ideal of the $\Ft$-module $\Ht{i}{X}$.

\end{defn}

Let $\vepsi:\pi_1(M)\rightarrow \Z$ be a fixed group homomorphism. Throughout this paper, we will assume that $\vepsi$ is an \textbf{epimorphism}. As we already pointed out, $\pi_1(M)$ is generated by a choice of positively oriented meridians $a_j$ around each hyperplane $H_j$. In fact, $H_1(M,\Z)\cong \Z^m$ is the free abelian group generated by the classes of those meridians. Hence, $\vepsi$ is completely determined by the value it takes in those oriented meridians. We will denote by $\vepsi_j:=\vepsi(a_j)$ for all $j=1,\ldots,m$. In this paper, we will study the twisted Alexander modules $\Ht{*}{M}$, and unless stated otherwise, we require $\vepsi$ to be a \textbf{positive} epimorphism, that is, $\vepsi_j>0$ for all $1\leq j\leq m$.

Throughout the paper, we will use the following notation.
\begin{notn}
Let $\gamma\in \pi_1(M)$. We denote by $\deti(\gamma)$ the determinant
$$
\det(t^{\vepsi(\gamma)}\rho(\gamma)-\Id)\in\Ft.
$$
\label{notndet}
\end{notn}

\subsection{Reidemeister torsion}
\label{sstorsion}

In \cref{sslines}, we will be using the \textbf{torsion} $\tau(C_*)$ of a finite chain complex $C_*$ of finite dimensional vector spaces over a field $\K$, as defined in \cite[Section 3]{milnor} (but we use multiplicative notation instead of additive notation, unlike in \cite{milnor}). The torsion $\tau(C_*)$ is an element of $\K^*/\{\pm 1\}$, and depends on a choice of bases for both the chain complex and its homology. In particular, if $C_*$ is acyclic, $\tau(C_*)$ only depends on a choice of bases for $C_*$. The actual definition of the torsion is not going to be relevant in this paper. The torsion behaves well with respect to short exact sequences, as illustrated in the following result.

\begin{lem}[\cite{milnor}]
Let
$$
0\longrightarrow C' \longrightarrow C \longrightarrow C'' \longrightarrow 0
$$
be a short exact sequence of based finite chain complexes of finite dimensional vector spaces, with compatible bases. Let $\cH$ be the associated long exact sequence in homology, viewed as a based acyclic complex, the bases being the fixed bases of the homology of $C'$, $C$, and $C''$. Then,
$$
\tau(C)=\tau(C')\tau(C'')\tau(\cH)
$$
where the torsion is taken with respect to the fixed bases.
\label{lemtorsion}
\end{lem}

Let $(X,\rho,\vepsi)$ be as in \cref{defntwistmod}. By tensoring $C^{\vepsi,\rho}_*(X,\Ft)$ with the field of rational functions $\F(t)$, we construct a finite chain complex of based finite dimensional vector spaces over $\F(t)$, which we call $C^{\vepsi,\rho}_*(X,\F(t))$.

\begin{defn}[{\cite[Section 3]{kirk}}]
We denote by $\tau_{\vepsi,\rho}(X)$ the \textbf{twisted Reidemeister torsion} of $(X,\vepsi,\rho)$, which is defined as
$$
\tau_{\vepsi,\rho}(X)=\tau(C^{\vepsi,\rho}_*(X,\F(t))).
$$
\label{defntwisttor}
\end{defn}

In this definition we have not specified a choice of bases of $C^{\vepsi,\rho}_*(X,\F(t))$, but we will only consider bases of the form $b\otimes c_i$, where $b$ is a basis of $\V$ as a vector space over $\F$ and $c_i$ is a ``geometric'' basis of $C_i(\widetilde X, \F)$ as a free left $\F[\pi_1(X)]$-module, that is, a basis obtained by lifting $i$-cells of $X$ for all $i$. We also have not specified a choice of bases of the homology of $C^{\vepsi,\rho}_*(X,\F(t))$, but in this paper we will only deal with the torsion of acyclic complexes, so we will not need to. The following result explains the indeterminacy of the torsion of such complexes.

\begin{lem}[{\cite[Section 3]{kirk}}]
Suppose that $C^{\vepsi,\rho}_*(X,\F(t))$ is acyclic. Then, $\tau_{\vepsi,\rho}(X)$ is independent of the choice of bases up to multiplication by a unit of $\Ft$.
\label{lemwelldef}
\end{lem}

In light of this last result, we will always consider $\tau_{\vepsi,\rho}(X)$ to be an element of $\F(t)$ up to multiplication by a unit of $\Ft$.

We end this section by stating the relation between the twisted Reidemeister torsion and the twisted Alexander polynomials.

\begin{lem}[{\cite[Theorem 3.4]{kirk}}]
Suppose that $C^{\vepsi,\rho}_*(X,\F(t))$ is acyclic, and let $\tau_{\vepsi,\rho}(X)$ be the twisted Reidemeister torsion of $(X,\vepsi,\rho)$. Then,
$$
\tau_{\vepsi,\rho}(X)=\frac{\prod_i\Delta_{2i+1}^{\vepsi,\rho}(X)}{\prod_i\Delta_{2i}^{\vepsi,\rho}(X)}
$$
up to multiplication by a unit of $\Ft$.
\label{lemtorpol}
\end{lem}

\subsection{Overview of the main results}

In this paper, we study the twisted Alexander modules and the twisted Alexander polynomials of both $M$ and a punctured stratified tubular neighborhood $W^*$ of $H$, as defined explicitly in \cref{tube}. The motivation behind studying $W^*$ and its relationship with $M$ comes from the fact that, roughly speaking, $W^*$ is constructed by gluing tubular neighborhoods of the different strata of the arrangement and then removing the arrangement, so it is well suited for Mayer-Vietoris type computations.

In \cref{shomotopytype}, we start by recalling a result from \cite{morse} (\cref{tocho} in this paper) involving a space $V$ which is the boundary of a certain neighborhood of the arrangement $H$ which fibers over a circle. This result will be very useful for proving the main result in that section, namely \cref{stratocho}, which relates the topologies of $M$ and $W^*$.

\begin{restatable*}{theorem}{stratocho}
$M$ has the homotopy type of $W^*$ with cells of dimension $\geq n$ attached.
\label{stratocho}
\end{restatable*}

In \cref{storsion} we study the torsion properties of the twisted Alexander modules of both $M$ and $W^*$, using the space $V$ and the fibration structure we know by \cref{tocho} to do so. The main result in this section is \cref{torsion}, which is the hyperplane arrangement extension of \cite[Theorem 4.1, Corollary 4.4]{maxtommy}. Note that \cite[Theorem 4.1, Corollary 4.4]{maxtommy} applies to Alexander modules of complements of hypersurfaces in general position at infinity, and hyperplane arrangements are not necessarily in general position at infinity. 

\begin{restatable*}{theorem}{torsion}
The twisted Alexander modules $\Ht{i}{M}$ are torsion $\Ft$-modules for every $0\leq i\leq n-1$, they are trivial modules for $i> n$, and $\Ht{n}{M}$ is a free $\Ft$-module of rank $(-1)^n\cdot\dim_\F(\V)\cdot\chi(M)$
\label{torsion}
\end{restatable*}

In the proof of this last result, and making use of \cref{stratocho}, we will also arrive at \cref{corroots}, which gives us a divisibility result.

\begin{restatable*}{cor}{corroots}
$\Ht{i}{M}$ and $\Ht{i}{W^*}$ are torsion $\Ft$-modules for any $0\leq i \leq n-1$. Moreover, their twisted Alexander polynomials $\Delta_i^{\vepsi,\rho}(M)$ and  $\Delta_i^{\vepsi,\rho}(W^*)$ coincide for $0\leq i <n-1$, and $\Delta_{n-1}^{\vepsi,\rho}(M)$ divides $\Delta_{n-1}^{\vepsi,\rho}(W^*)$.
\label{corroots}
\end{restatable*}

Finally, in \cref{sroots} we study the twisted Alexander polynomials of $W^*$, which also give us information about the twisted Alexander polynomials of $M$ by the divisibility result of \cref{storsion}, namely \cref{corroots}. This is easier to do in the case of line arrangements (\cref{sslines}), where we are able to find an explicit formula for $\Delta_1^{\vepsi,\rho}(W^*)$ as in the following result.

\begin{restatable*}{theorem}{alexlines}
Let $\cA=\{H_1,\ldots,H_m\}\subset \C^2$ be an essential line arrangement, $H=\bigcup\limits_{i=1}^m H_i$, $M=\C^2\backslash H$ and let $W^*$ be a punctured stratified tubular neighborhood of $H$. Let $P_1,\ldots,P_s$ be the singular points of $H$, let $s_i$ be the number of singular points of $H$ on $H_i$, and let $d_k$ be the number of lines of $\cA$ going through the singular point $P_k$. Let $a_1,\ldots,a_m$ be the generators of $\pi_1(M)$ as described in \cref{ssdef}, and $\beta_k$ as described in \cref{defbeta}. Then, following \cref{notndet}, we have
\begin{enumerate}
\item $\Delta_1^{\vepsi,\rho}(W^*)=\left(\prod\limits_{k=1}^s \deti(\beta_k)^{d_k-2}\right)\cdot\left(\prod\limits_{i=1}^m\deti(a_i)^{s_i-1}\right)\cdot\Delta_0^{\vepsi,\rho}(M)$.
\item $\Delta_1^{\vepsi,\rho}(M)$ divides
$$\left(\prod_{k=1}^s \deti(\beta_k)^{d_k-2}\right)\cdot\left(\prod_{i=1}^m\deti(a_i)^{s_i-1}\right)\cdot\gcd_{i=1,\ldots,m}\left\{\deti(a_i)\right\}.$$
\label{div}
\end{enumerate}
\label{alexlines}
\end{restatable*}

Note that \cref{corroots} and \cref{alexlines} refine the divisibility result obtained in \cite[Theorem 1.1]{cogo} (where $\rho$ is further assumed to be unitary), since the polynomial which $\Delta_1^{\vepsi,\rho}(M)$ divides in \cite[Theorem 1.1]{cogo} also contains contributions from the link at infinity. In some cases, we will be able to further refine the bound for $\Delta_1^{\vepsi,\rho}(M)$ given by \cref{alexlines}, part \ref{div}, as shown in the other main result of \cref{sslines}.

\begin{restatable*}{theorem}{linestube}
Let $\cB$ be the set of lines in $\cA=\{H_1,\ldots,H_m\}$ such that for each line in $\cB$ no other line in $\cA$ is parallel to it. Suppose that $\cB=\{H_1,\ldots, H_l\}\neq\emptyset$. Let $s_i$ be the number of singular points of $H$ in $H_i$, which we denote by $P_1^i,\ldots,P_{s_i}^i$, and let $d_k^i$ be the number of lines of $\cA$ going through the singular point $P_k^i$. Let $a_1,\ldots,a_m$ be the generators of $\pi_1(M)$ as described in \cref{ssdef}, and let $\beta_{k,i}$ be the resulting loop from composing all of the meridian loops around lines in $\cA$ going through the singular point $P_k^i$ as in \cref{defbeta} (and \cref{remHopf}). Then, $\Delta_1^{\vepsi,\rho}(M)$ divides
$$
\left(\gcd_{r=1,\ldots, m}\{\deti(a_r)\}\right)\cdot\gcd_{i=1,\ldots,l}\left\{\left(\prod_{k=1}^{s_i} \deti(\beta_{k,i})^{d_k^i-2}\right)\cdot\deti(a_i)^{s_i-1}\right\}.
$$
\label{linestube}
\end{restatable*}

In the higher dimensional case discussed in \cref{ssgeneralcase}, we use the natural stratification of $H$ to obtain an open cover of $W^*$, namely 
$$\{\mathcal{S}_l^k \mid k=0,\ldots,n-1; l=1,\ldots,s_k\}$$
such that each one of the open sets in the cover fibers over the corresponding stratum $\Sigma_l^k$ of dimension $k$ of $H$. Then, we use the Mayer-Vietoris cohomology spectral sequence for the twisted Alexander modules associated to this open cover to get a bound for the twisted Alexander polynomials $\Delta_i^{\vepsi,\rho}(M)$, and arrive at the following result, which generalizes \cref{alexlines}.

\begin{restatable*}{theorem}{rootsgeneral}
Let $\mathcal{A}=\{H_1,\ldots,H_m\}$ be an essential hyperplane arrangement in $\C^n$, with the natural induced stratification $\{\Sigma_l^k \mid k=0,\ldots,n-1; l=1,\ldots,s_k\}$, and let $M$ be the complement of that arrangement in $\C^n$. For every $k$ and $l$, let $F_{l,k}$ be the fiber of the fibration $\cS_{l}^k\longrightarrow \Sigma_l^k$ and let $\gamma_\infty(F_{l,k})$ be a meridian around the hyperplane at infinity in $\C\P^{n-k}$ with positive orientation, where $F_{l,k}$ is naturally seen in $\C\P^{n-k}$. Then, for any $i=0,\ldots,n-1$, the zeros of the $i$-th Alexander polynomial of $M$ (i.e. $\Delta_i^{\vepsi,\rho}(M)$) are among those of
$$
\prod_{k=0}^{n-1}\prod_{l=1}^{s_k}\deti(\gamma_\infty(F_{l,k})).
$$
\label{rootsgeneral}
\end{restatable*}

\section{The homotopy type of M}
\label{shomotopytype}

We will study the topology of $M$ with the help of two functions ($f_\vepsi$ and $g_\vepsi$) defined from the fixed positive epimorphism $\vepsi:\pi_1(M)\rightarrow \Z$ as follows:
$$
\f{f_\vepsi}{M}{\R}{z}{\displaystyle\prod_{j=1}^m \mid \xi_j(z)|^{\vepsi_j}}
$$
$$
\f{g_\vepsi}{M}{\R/2\pi\Z\cong S^1}{z}{\displaystyle\arg\left(\prod_{j=1}^m \xi_j(z)^{\vepsi_j}\right)=\displaystyle\sum_{j=1}^m \vepsi_j\cdot \arg(\xi_j(z))}
$$

Let $\delta>0$ small enough, let $V:= f_{\vepsi}^{-1}(\delta)$. The following result can be found in \cite[Theorem 2.3]{morse}. In the original statement of this theorem in \cite{morse}, one of the hypotheses is that $(\vepsi_1,\ldots,\vepsi_m)$ is of ``rank 1'', that is, that $\{\vepsi_1,\ldots,\vepsi_m\}$ span a $1$-dimensional $\Q$-vector space over $\R$, which is automatically satisfied in our case.% since $\vepsi$ is an epimorphism onto $\Z$.
\\
\begin{theorem}
For every $\delta>0$ small enough, we have that
\begin{enumerate}
\item $V$ is a smooth manifold of dimension $2n-1$.
\item The inclusion $V\into f_{\vepsi}^{-1}\left((0,\delta]\right)$ is a homotopy equivalence. \label{tochohom}
\item The map ${g_\vepsi}_{| V}:V\rightarrow S^1$ is a fiber bundle, and the fiber $F$ has the homotopy type of a finite CW-complex of dimension $n-1$.\label{fiber}
\item $M$ has the homotopy type of $V$ with $|\chi(M)|$ cells of dimension $n$ attached.\label{hom}
\end{enumerate}
\label{tocho}
\end{theorem}

\begin{rem}
Note that the space $V$ depends on both $\delta$ and the homomorphism $\vepsi$.
\label{remdependence}
\end{rem}

\cref{tocho} gives us some good properties of $f_{\vepsi}^{-1}\left((0,\delta]\right)$, which we will be using in \cref{storsion}. However, those properties alone will not be enough for us to compute possible roots of the twisted Alexander polynomials of $M$. The rest of this section is devoted to describe a different neighborhood of the arrangement with a nice stratification and prove some properties about it that will come in handy in \cref{sroots}.

We stratify our hyperplane arrangement in the natural way: two points $P$ and $P'$ in $H$ lie in the same stratum if the collections of hyperplanes in the arrangement containing $P$ and $P'$ coincide. Each stratum is a smooth submanifold of $\C^n$. We define a neighborhood $W$ of $H$ inductively as follows. Let $\Sigma_k$ the union of strata of dimension $k$ in $H$. For each stratum of dimension $0$, we pick a ball of radius $\delta_0$ around it, and call $W(\delta_0)$ the union of those balls. Now, we take a tubular neighborhood of $\Sigma_1\backslash \overline{W\left(\frac{\delta_0}{2}\right)}$ of radius $\delta_1<\delta_0$, and define $W(\delta_0,\delta_1)$ as the union of $W(\delta_0)$ with this tubular neighborhood that we have just described. Now, we take a tubular neighborhood of $\Sigma_2\backslash \overline{W\left(\frac{\delta_0}{2},\frac{\delta_1}{2}\right)}$ of radius $\delta_2<\delta_1$ and create $W(\delta_0,\delta_1,\delta_2)$. We proceed inductively until we reach $W:=W(\delta_0,\ldots,\delta_{n-1})$.

Note that, when all of the $\delta$'s are small enough, all of these neighborhoods that we have defined are homeomorphic. \textbf{From now on, we will assume that all of the $\delta$'s are small enough, and will not specify them.}

\begin{defn}
We call $W$ a stratified tubular neighborhood of $H$. We let $W^*=W\backslash H$, and call it a \textbf{punctured stratified tubular neighborhood of $H$}.
\label{tube}
\end{defn}

\begin{rem}
$W^*$ is homotopy equivalent to $\partial W$.%, when seeing $W$ as a subset of $\C^n$.
\label{rmkboundary}
\end{rem}

The following theorem relates the topologies of $W^*$ and $M$.

\stratocho

The proof of this theorem is an immediate consequence of the following proposition.

\begin{prop}
Let $j:W^*\hookrightarrow M$ be the inclusion. Then
\begin{enumerate}
    \item $j_*:\pi_i(W^*)\longrightarrow \pi_i(M)$ is an isomorphism for $i<n-1$.
    \item $j_*:\pi_{n-1}(W^*)\longrightarrow \pi_{n-1}(M)$ is an epimorphism.
\end{enumerate}
\label{prophomotopy}
\end{prop}

\begin{proof}
The outline of the proof is going to be the following. First, we will find two stratified tubular neighborhoods $W$ and $W'$ of $H$ and a $\delta >0$ such that $W'\subset f_\vepsi^{-1}([0,\delta])\subset W$. Then, we will get the result about $W^*$ from the information about $f_{\vepsi}^{-1}\left((0,\delta]\right)$ that we know from \cref{tocho} and the fact that the inclusion $W'\backslash H\hookrightarrow W\backslash H$ is a homotopy equivalence.

Let us start with a stratified tubular neighborhood $W=W(\delta_0,\ldots,\delta_{n-1})$, and let $\delta'$ be the minimum of the $\delta_i$'s. We have that every point that is at distance less than $\delta'$ of $H$ is contained in $W$. Also note that the factors defining $f_{\vepsi}(z)$ are all proportional to a positive power of the distance of a point z to the hyperplane defined by that factor. Hence, for sufficiently small $\delta$, $f_\vepsi^{-1}([0,\delta])$ will be contained in the set of points on $\C^n$ that are at distance less than $\delta'$ of $H$, which is in turn contained in $W$. Thus, we have found $\delta$ such that $f_\vepsi^{-1}([0,\delta])\subset W$.

Let us find a stratified tubular neighborhood $W'$ of $H$ such that $W'\subset f_\vepsi^{-1}([0,\delta])$ to complete the first part of our outline of the proof. This $W'$ is constructed by taking the union of tubular neighborhoods of open sets of the strata like in \cref{tube}, but not requiring those tubular neighborhoods to have a fixed radius. These ``generalized'' stratified tubular neighborhoods are still homotopy equivalent to the ones in \cref{tube}. It is straightforward to see that we can find one such $W'$ inside of $f_\vepsi^{-1}([0,\delta])$. In particular, we have that $W'\backslash H\subset f_{\vepsi}^{-1}\left((0,\delta]\right)\subset W\backslash H=W^*$. 

Let us look at the following diagram, where all of the arrows are induced by inclusions.

\begin{equation}
\pi_i(W'\backslash H)\overset{a_i}{\longrightarrow} \pi_i(f_{\vepsi}^{-1}\left((0,\delta]\right))\overset{b_i}{\longrightarrow} \pi_i(W^*)\overset{c_i}{\longrightarrow}\pi_i(M)
\label{eqhomotopy}
\end{equation}

Since the inclusion from $W'\backslash H$ to $W^*$ is a homotopy equivalence, we have that $b_i\circ a_i$ is an isomorphism for all $i$. In particular, $b_i$ is an epimorphism for all $i$. Also, by \cref{tocho}, parts \ref{tochohom} and \ref{hom}, we have that $c_i\circ b_i$ is an isomorphism if $i<n-1$ and an epimorphism if $i=n-1$. In particular, $b_i$ is a monomorphism if $i<n-1$, and $c_i$ is an epimorphism for $i\leq n-1$. This concludes the proof of the second assertion of the proposition.

Since we already know that $b_i$ is an epimorphism for all $i$ and a monomorphism if $i<n-1$, we find that $b_i$ is an isomorphism if $i<n-1$. Since $c_i\circ b_i$ is an isomorphism for $i<n-1$, we get that $c_i$ is an isomorphism for $i<n-1$, and this concludes the proof of the first assertion of the proposition.
\end{proof}

\section{Torsion properties of the twisted Alexander modules}\label{storsion}

From now on, we fix $\delta>0$ small enough so that \cref{tocho} holds. Let $j:V\hookrightarrow M$ be the inclusion, and $j_*:\pi_1(V)\rightarrow \pi_1(M)$ be the map it induces on fundamental groups. Abusing notation, we will also denote by $\vepsi$ and $\rho$ the induced maps on $\pi_1(V)$ that we get by composing $j_*$ with $\vepsi$ and $\rho$ respectively.
\\
\begin{prop}
Let $n\geq 2$. The inclusion map $j:V\hookrightarrow M$ induces isomorphisms of $\Ft$-modules
$$
\Ht{i}{V}\xrightarrow{\cong}\Ht{i}{M}
$$
for any $i<n-1$, and an epimorphism of $\Ft$-modules
$$
\Ht{n-1}{V}\twoheadrightarrow\Ht{n-1}{M}.
$$
\label{dom}
\end{prop}

\begin{proof}
We consider two cases: $n>2$ and $n=2$.

Suppose that $n>2$. By \cref{tocho}, part \ref{hom}, the space $M$ is obtained from $V$ by attaching cells of dimension $n\geq 3$, so $j_*$ is an isomorphism of fundamental groups. Hence, the chain complexes $C_*^{\vepsi,\rho}(V,\Ft)$ and $C_*^{\vepsi,\rho}(M,\Ft)$ are the same from place $n-1$ down, and $j$ induces an inclusion $C_n^{\vepsi,\rho}(V,\Ft)\hookrightarrow C_n^{\vepsi,\rho}(M,\Ft)$. The result follows from this observation.

Now, let us consider the case $n=2$. In this case, applying \cref{tocho}, part \ref{hom}, only tells us that $j_*$ is an epimorphism between the fundamental groups. We have that $\ker j_*\subset \ker \vepsi\circ j_*$ is a normal subgroup of $\pi_1(V)$. Let $V_{\ker j_*}$ be the covering space associated to $\ker j_*$, and note that $\pi_1(V)/\ker j_*\cong \pi_1(M)$.

We construct the chain complex
$$
D_*:=(\Ft\otimes_\F \V))\otimes_{\F[\pi_1(V)/\ker j_*]}C_*(V_{\ker j_*},\F).
$$
The inclusion $V\hookrightarrow M$ induces a map $V_{\ker j_*}\rightarrow \widetilde M$, where $\widetilde M$ is the universal cover of $M$. Since the space $M$ is obtained from $V$ by attaching cells of dimension $\geq 2$, this map induces isomorphisms
$$
D_i\xrightarrow{\cong} C_i^{\vepsi,\rho}(M,\Ft)=(\Ft\otimes_\F \V))\otimes_{\F[\pi_1(M)]}C_i(\widetilde M,\F)
$$
for $i=0,1$ and a monomorphism
$$
D_2\hookrightarrow C_2^{\vepsi,\rho}(M,\Ft).
$$

Thus, we have an isomorphism
$$
H_0(D_*)\xrightarrow{\cong}\Ht{0}{M}
$$
and an epimorphism
$$
H_1(D_*)\twoheadrightarrow\Ht{1}{M}
$$

By \cite[Section 2.5, p. 50]{dimca}, the homology of $D_*$ is the same as the homology of $C_*^{\vepsi,\rho}(V,\Ft)$. The result follows from this observation.
\end{proof}

\begin{rem}
Using the discussion following diagram (\ref{eqhomotopy}) in the proof of \cref{prophomotopy}, and repeating the same steps in the proof of \cref{dom}, we can conclude that the same results hold for the maps $\Ht{i}{V}\longrightarrow \Ht{i}{W^*}$ and $\Ht{i}{W^*}\longrightarrow \Ht{i}{M}$ induced by inclusion.
\label{remdom}
\end{rem}

The following corollary is a direct consequence of \cref{dom} and \cref{remdom}.

\begin{cor}
Let $n\geq 2$. For any $0\leq i \leq n-1$, if $\Ht{i}{V}$ is a torsion $\Ft$-module, then so are $\Ht{i}{M}$ and $\Ht{i}{W^*}$.
\label{cor}
\end{cor}

Now, we will show that the hypothesis of \cref{cor} is actually satisfied.
\begin{theorem}
Let $n\geq 2$. Then, $\Ht{i}{V}$ is a torsion $\Ft$-module for all $i$.
\label{uve}
\end{theorem}
\begin{proof}
Note that $(g_\vepsi)_*=\vepsi$. Let $V^{\vepsi}\xrightarrow{p_1} V$ be the covering space induced by $\ker{\vepsi}$. Recall that by \cref{tocho}, part \ref{fiber}, the map $(g_{\vepsi})_{| V}:V\rightarrow S^1$ is a fiber bundle. We call the fiber $F$.

The covering space $V^{\vepsi}\xrightarrow{p_1} V$ is the pullback by $(g_{\vepsi})_{| V}$ of the universal cover $\R\xrightarrow{p_2} S^1$, and we have the following commutative diagram of the pullback
\begin{center}
\begin{tikzcd}
V^{\vepsi}\arrow[d, "p_1"]\arrow[r] & \R\arrow[d,"p_2"]\\
V
\arrow[r,"g_{\vepsi}"]
& S^1
\end{tikzcd}
\end{center}

Note that $V^{\vepsi}\rightarrow \R$ is a fiber bundle over a contractible space with fiber $F$, so $V^{\vepsi}$ is homeomorphic to $F\times \R$, and therefore homotopically equivalent to $F$.

Let $\mathcal{L}_\rho$ be the local system of $\F$-vector spaces given by the representation of $\pi_1(V^\vepsi)$ induced by $\rho$. By \cite[Theorem 2.1]{kirk}, we have that
$$
\Ht{i}{V}\cong H_i(V^\vepsi,\mathcal{L}_\rho)
$$
as $\Ft$-modules for all $i$. Since $V^\vepsi$ is homotopy equivalent to $F$, which by \cref{tocho}, part \ref{fiber}, has the homotopy type of a finite CW-complex, we have that the $H_i(V^\vepsi,\mathcal{L}_\rho)$ are finite dimensional $\F$-vector spaces for all $i$, and thus $\Ht{i}{V}$ are torsion $\Ft$-modules for all $i$.
\end{proof}

Let us recall the following fact, which can be found in \cite{kirk}.

\begin{prop}
Let $X$ be a finite CW-complex. If $\vepsi$ is non-trivial, then $$\Ht{0}{X}$$
is a torsion $\Ft$-module.
\label{pzero}
\end{prop}

Now, we are ready to prove the main result in this section.
\torsion
\begin{proof}
The space $M$ is an affine variety of complex dimension $n$, so it is homotopy equivalent to a finite CW complex of real dimension $n$ (\cite{dimca2,milmor}). Thus $\Ht{i}{M}=0$ for $i>n$. This also implies that $\Ht{n}{M}$ is a free module, since it is the kernel of a morphism of free $\Ft$-modules.

Now, let us prove that the twisted Alexander modules $\Ht{i}{M}$ are torsion $\Ft$-modules for every $0\leq i\leq n-1$. If $n=1$, this is true by \cref{pzero}. Suppose that $n\geq 2$. In that case, by \cref{cor}, we just need to show that $\Ht{i}{V}$ is a torsion $\Ft$-module for every $0\leq i\leq n-1$, which is true by \cref{uve}.

Finally, let us compute the rank of $\Ht{n}{M}$. We abuse notation and call $\mathcal{L}_{\vepsi,\rho}$ the local system of vector spaces over the field of rational functions $\F(t)$ defined by the tensor representation induced by $\vepsi$ and  $\rho$ (instead of the local system of $\Ft$-modules induced by $\vepsi$ and $\rho$). By \cite[Proposition 2.5.4]{dimca}, we have that
$$
(-1)^n\rank_{\Ft}\Ht{n}{M}=\chi(M,\mathcal{L}_{\vepsi,\rho})=\rank_{\Ft}(\Ft\otimes_\F \V)\cdot\chi(M)=\dim_\F(\V)\cdot\chi(M).
$$
Hence,
$$
\rank_{\Ft}\Ht{n}{M}=(-1)^n\cdot\dim_\F(\V)\cdot\chi(M).
$$
\end{proof}

\begin{rem}
The space $V$ depends on the epimorphism $\vepsi$, but $W^*$ does not. This dependence on $\vepsi$ came in handy in the proof of \cref{torsion}, although it can be proved that the Alexander modules $\Ht{i}{W^*}$ are torsion for all $i\geq 0$ directly, as we will see in \cref{ssgeneralcase}.
\end{rem}

We end this section with the result that we will use in \cref{sroots}, which is a consequence of everything we have discussed in this section.

\corroots

\section{Roots of twisted Alexander polynomials}\label{sroots}

%In this section, we will compute the twisted Alexander polynomials of $W^*$ (or of $\partial W$ in the line arrangement case, see Remark \ref{rmkboundary}). By Corollary \ref{corroots}, this will give us a set in which the roots of the twisted Alexander polynomials of the arrangement complement $M$ lie.

\subsection{Line arrangement case ($n=2$)}
\label{sslines}
Let $\mathcal{A}=\{H_1,\ldots,H_m\}\subset \C^2$ be an essential line arrangement. Note that, in the line arrangement case, the only two twisted Alexander polynomials that we will be considering are the $0$-th and the first ones.

The $0$-th case is always easy to compute, not just in dimension $2$. The $0$-th and first twisted Alexander polynomials of any finite CW complex can be computed from a presentation of the fundamental group using Fox Calculus (\cite[Section 4]{kirk}). In particular, if $M$ is the complement of a complex hyperplane arrangement $\{H_1,\ldots,H_m\}$, after a Lefschetz type argument we can in principle use a presentation of $\pi_1(M)$ to compute the $0$-th and first twisted Alexander polynomials of $M$. Let us consider the map of $\Ft$-modules
$$
\partial:(\Ft\otimes \V)^m\rightarrow \Ft\otimes \V
$$
given by the column matrix with entries $$t^{\vepsi(a_i)}\rho(a_i)-\id\in \cM_{\dim_\F\V\times \dim_\F\V}(\Ft),\text{ \ \ \ }i=1,\ldots,m$$
where $a_1,\ldots,a_m$ are the generators of $\pi_1(M)$ as described in \cref{ssdef}. The $0$-th twisted Alexander polynomial $\Delta_0^{\vepsi,\rho}(M)$ is just a generator of the Fitting ideal of the cokernel of $\partial$, so it is the greatest common divisor of the minors of size $\dim_\F\V$ of the column matrix we just described (see \cite[Section 4]{kirk}). Hence, we have the following result.

\begin{prop}
$\Delta_0^{\vepsi,\rho}(M)$ is the greatest common divisor of the minors of size $\dim_\F\V$ of the column matrix with entries $$t^{\vepsi(a_i)}\rho(a_i)-\id\in \cM_{\dim_\F\V\times \dim_\F\V}(\Ft)$$
for $i=1,\ldots, m$. In particular, using \cref{notndet},
$$
\Delta_0^{\vepsi,\rho}(M)\text{ divides }\gcd_{i=1,\ldots,m}\left\{\deti(a_i)\right\}
$$
\label{zero}
\end{prop}

Now, let us study the first twisted Alexander polynomials of $M$. We have the following result.

\alexlines

\begin{proof}
We will use techniques coming from \cite[Theorem 5.6]{cogo}, although, in our case, the work done in Sections \ref{shomotopytype} and \ref{storsion} allows us to not have to deal with contributions coming from the line at infinity, unlike in \cite{cogo}. Let $F=H\backslash \bigsqcup\limits_{k=1}^s (H\cap \B_k^4)$ be the surface obtained by removing small balls $\B_k^4$ around the singular points $P_k$. Note that what we are really removing from our surface is a $2$-dimensional open disk $D_i^k$ from every line $H_i$ in $\cA$ containing $P_k$.

Let $N=F\times S^1$. $N$ should be thought of as the boundary of a tubular neighborhood around the non-singular part of $H$. We have that $\partial N=\partial F\times S^1$, and since $\partial F$ is a union of disjoint $S^1$'s (one from every disk $D_i^k$ removed), then $\partial N$ is a union of disjoint tori $\bigsqcup\limits_{k,i} T_i^k$ (again, one from every disk $D_i^k$ removed). Let us fix a point $f_k^i$ in the $S^1$ corresponding to the boundary of the disk $D_i^k$ for every such disk removed.

Let $L_k$ be the link of the singularity at the point $P_k$ (which is a Hopf link with $d_k$ components), and let $S^3_k$ be the boundary of $\B_k^4$. We consider the space
$$
X=N\cup_{\left(\bigsqcup\limits_{k,i} T_i^k\right)} \left(\bigsqcup_{k=1}^s S^3_k\backslash L_k\right)\subset M
$$
where the gluing is done as follows. A meridian around the $i$-th component of $L_k$ (the one corresponding to the line $H_i$, which we will denote by $L_k^i$) is glued to $\{f_k^i\}\times S^1\subset N$, and $L_k^i$ is glued to the $S^1$ corresponding to the boundary of $D_i^k$.

By the definition of the stratified tubular neighborhood $W$, we have that $X$ is homotopy equivalent to $\partial W$. By \cref{corroots}, the first twisted Alexander polynomial of the line arrangement complement $M$ divides the first twisted Alexander polynomial of $\partial W$ (which is homotopy equivalent to $W^*$, see \cref{rmkboundary}), so our goal now is to compute $\Delta_1^{\vepsi,\rho}(X)$.

Notice that $N$ has $m$ connected components, one for every line $H_i$ in our arrangement. That is, if we define $F_i=F\cap H_i$, then $N=\bigsqcup\limits_{i=1}^m F_i\times S^1$. Notice that $F_i$ is just a complex line $H_i$ (or a real plane) with $s_i$ disks removed, one for every singular point of $H$ in $H_i$. Thus, $F_i$ is homotopy equivalent to a wedge sum of $s_i$ circles, and hence $F_i\times S^1$ (and $N$) is homotopy equivalent to a $2$-dimensional CW-complex.

It is also well-known (\cite[Lemma 2]{lib2}) that $S^3_k\backslash L_k$ has the homotopy type of a $2$-dimensional CW-complex as well. The space $X$ also has the homotopy type of a $2$-dimensional CW-complex by how it is constructed.

We have the following Mayer-Vietoris short exact sequence of complexes with coefficients in $\F(t)$.
$$
0\rightarrow\bigoplus_{k,i}C_*^{\vepsi,\rho}(T_i^k,\F(t))\rightarrow\left(\bigoplus_k C_*^{\vepsi,\rho}(S^3_k\backslash L_k,\F(t))\right)\oplus C_*^{\vepsi,\rho}(N,\F(t))\rightarrow C_*^{\vepsi,\rho}(X,\F(t))\rightarrow 0
$$

Let $\cH$ be the Mayer-Vietoris long exact sequence of the twisted homology groups (seen as a complex). We will consider the twisted Reidemeister torsion $\tau_{\vepsi,\rho}$ (as defined in \cref{defntwisttor}) of all the pieces involved in this short exact sequence, namely $N$, $\bigsqcup\limits_{k=1}^s S^3_k\backslash L_k$, and their intersection $\bigsqcup\limits_{k,i} T_i^k$.

As pointed out in  \cref{lemwelldef}, the twisted Reidemeister torsion for acyclic complexes is independent of the choice of bases up to multiplication by a unit in $\Ft$, and, as we will see in the proof of \cref{H}, we only consider the twisted Reidemeister torsion of acyclic complexes in this proof. Since all of those pieces (including $X$) have the homotopy type of a $2$-dimensional CW-complex, then the only non-trivial Alexander polynomials are the $0$-th and the first ones for all of those spaces, and by \cref{lemtorpol}, we have that
$$
\tau_{\vepsi,\rho}(\cdot)=\frac{\Delta_1^{\vepsi,\rho}(\cdot)}{\Delta_0^{\vepsi,\rho}(\cdot)}
$$
for all of the relevant spaces in this problem ($X$, $N$, $\bigsqcup\limits_{k,i} T_i^k$ and $\bigsqcup\limits_{k=1}^s S^3_k\backslash L_k$).

By \cref{lemtorsion}, we have that
\begin{equation}
\left(\prod_{k=1}^s\tau_{\vepsi,\rho}(S^3_k\backslash L_k)\right)\left(\tau_{\vepsi,\rho}(N)\right)=\left(\prod_{k,i}\tau_{\vepsi,\rho}(T_i^k)\right)\tau_{\vepsi,\rho}(X)\tau(\cH)
\label{tor}
\end{equation}
were $\tau(\cH)$ is the torsion of a complex. 

Now, we use the following result.

\begin{prop}
$\cH$ is the trivial complex. In particular, $\tau(\cH)=1$
\label{H}
\end{prop}
\begin{proof}
We need to show that the complexes $C_*^{\vepsi,\rho}(T_i^k,\F(t))$ (for every $k$ and $i$), $C_*^{\vepsi,\rho}(S^3_k\backslash L_k,\F(t))$ (for every $k$), $C_*^{\vepsi,\rho}(N,\F(t))$ and $C_*^{\vepsi,\rho}(X,\F(t))$ are acyclic. By the long exact sequence in homology, it suffices to show that three out of those four are acyclic. 

By \cite[Proposition 2.9]{maxtommy}, since  $\F(t)$ is flat over $\Ft$ and $\vepsi(\beta_k)\neq 0$ (in fact, $\vepsi(\beta_k)> 0$ by \cref{remHopf}), we have that $$C_*^{\vepsi,\rho}(S^3_k\backslash L_k,\F(t))$$ is acyclic.

Let us now see that $C_*^{\vepsi,\rho}(N,\F(t))$ is acyclic, or equivalently, that $H_j^{\vepsi,\rho}(F_i\times S^1,\F(t))=0$ for all $i=1,\ldots,m$ and $j\geq 0$. We can compute $H_0^{\vepsi,\rho}(F_i\times S^1,\F(t))$ and $H_1^{\vepsi,\rho}(F_i\times S^1,\F(t))$ directly using Fox Calculus (\cite[Section 4]{kirk}), a technique that only requires a presentation of the fundamental group. Recall that $F_i$ is homotopy equivalent to a wedge sum of $s_i$ circles, and let $b_1^i,\ldots,b_{s_i}^i$ be loops around the respective circles. With this notation, we see that
\begin{equation}
\pi_1(F_i\times S^1)=\langle b_1^i,\ldots,b_{s_i}^i,a_i \mid [b_j^i,a_i], j=1,\ldots,s_i\rangle.
\label{eqnFi}
\end{equation}

In this presentation we are abusing notation, since the base point of $a_i$ is not in $F_i\times S^1$. By $a_i$ in this presentation, we mean a loop contained in $F_i\times S^1$ that is isotopic to $a_i$ in $M$ after a change of base points.

Using that $\vepsi(a_i)$ is not $0$ for any $i=1,\ldots,m$ in a routine Fox Calculus computation using this presentation, we get that $H_j^{\vepsi,\rho}(F_i\times S^1,\F(t))=0$ for $j=0,1$.

To finish proving that $H_j^{\vepsi,\rho}(F_i\times S^1,\F(t))=0$ for all $j$, we just have to show it for $j=2$, since $F_i\times S^1$ is homotopy equivalent to a $2$-dimensional CW-complex. This $2$-dimensional CW-complex is the cartesian product of a wedge sum of $s_i$ $S^1$'s and an $S^1$. Thus, it has one $0$-cell, ($s_i+1$) $1$-cells, and $s_i$ $2$-cells. Hence, an Euler characteristic argument tells us that $H_2^{\vepsi,\rho}(F_i\times S^1,\F(t))=0$, concluding our proof of the acyclicity of $C_*^{\vepsi,\rho}(N,\F(t))$.

The only thing left to prove here is that $C_*^{\vepsi,\rho}(T_i^k,\F(t))$ is acyclic (for every $k$ and $i$). This is just a computation that follows the same steps as what we did for $C_*^{\vepsi,\rho}(F_i\times S^1,\F(t))$, so we will omit it. It also relies on the fact that $\vepsi(\gamma_i)\neq 0$, for every meridian $\gamma_i$ around $H_i$ and for all $i=1,\ldots, m$.
\end{proof}

Now, using this result, equation (\ref{tor}) becomes
\begin{equation}
\left(\prod_{k=1}^s\tau_{\vepsi,\rho}(S^3_k\backslash L_k)\right)\left(\tau_{\vepsi,\rho}(N)\right)=\left(\prod_{k,i}\tau_{\vepsi,\rho}(T_i^k)\right)\tau_{\vepsi,\rho}(X).
\label{tor2}
\end{equation}

We want to compute $\Delta_1^{\vepsi,\rho}(X)$. By \cref{dom}, we have that $\Delta_0^{\vepsi,\rho}(X)=\Delta_0^{\vepsi,\rho}(M)$, and we know $\Delta_0^{\vepsi,\rho}(M)$ by \cref{zero}. Hence, to compute $\Delta_1^{\vepsi,\rho}(X)$, it suffices to compute $\tau_{\vepsi,\rho}(X)=\frac{\Delta_1^{\vepsi,\rho}(X)}{\Delta_0^{\vepsi,\rho}(X)}$. By the equation relating the torsions that we just found, it suffices to compute the twisted Reidemeister torsion for the other pieces.

\begin{prop}
\ 
\begin{enumerate}
\item $\tau_{\vepsi,\rho}(N)=\prod\limits_{i=1}^m\deti(a_i)^{s_i-1}$
\item $\tau_{\vepsi,\rho}\left(\bigsqcup\limits_{k,i} T_i^k\right)=1$
\item $\tau_{\vepsi,\rho}\left(\bigsqcup\limits_{k=1}^s S^3_k\backslash L_k\right)=\prod\limits_{k=1}^s \deti(\beta_k)^{d_k-2}$
\end{enumerate}
\label{N}
\end{prop}

\begin{proof}
First of all, by the multiplicativity of the torsion (which can be inferred from \cref{lemtorpol}), we have that
$$
\begin{array}{lcr}
\tau_{\vepsi,\rho}(N)=\prod\limits_{i=1}^m \tau_{\vepsi,\rho}(F_i\times S^1), & & \tau_{\vepsi,\rho}\left(\bigsqcup\limits_{k,i} T_i^k\right)=\prod\limits_{k,i} \tau_{\vepsi,\rho}(T_i^k).\\
\end{array}
$$
Using the presentations given in equations (\ref{eqnLink}) and (\ref{eqnFi}) and Fox Calculus (\cite[Section 4]{kirk}), we can compute the twisted Reidemeister torsion of all the spaces involved, namely
$$
\begin{array}{lcll}
\tau_{\vepsi,\rho}(F_i\times S^1)=\deti(a_i)^{s_i-1} & \text{ \ \ \ \ \ \ \ \ \ \ \ }& \text{for all }i=1,\ldots,m& \\
\tau_{\vepsi,\rho}(T_i^k)=1 & \text{ \ \ \ \ \ \ \ \ \ \ \ }& \text{for all }k,i& \\
\tau_{\vepsi,\rho}(S^3_k\backslash L_k)=\deti(\beta_k)^{d_k-2}& \text{ \ \ \ \ \ \ \ \ \ \ \ }& \text{for all }k=1,\ldots,s &\text{ \cite[Proposition 2.9]{maxtommy}}.\\ 
\end{array}
$$
\end{proof}

Now, we can use \cref{N} and equation (\ref{tor2}) to get
$$
\tau_{\vepsi,\rho}(X)=\prod_{k=1}^s \deti(\beta_k)^{d_k-2}\prod_{i=1}^m\deti(a_i)^{s_i-1}
$$
where this equality is defined up to multiplication by a unit of $\Ft$.

Hence
$$
\Delta_1^{\vepsi,\rho}(X)=\left(\prod_{k=1}^s \deti(\beta_k)^{d_k-2}\right)\cdot\left(\prod_{i=1}^m\deti(a_i)^{s_i-1}\right)\cdot\Delta_0^{\vepsi,\rho}(M)
$$
so, $\Delta_1^{\vepsi,\rho}(X)$ divides $$\left(\prod_{k=1}^s \deti(\beta_k)^{d_k-2}\right)\cdot\left(\prod_{i=1}^m\deti(a_i)^{s_i-1}\right)\cdot\gcd_{i=1,\ldots,m}\left\{\deti(a_i)\right\}.$$
Now, by \cref{corroots} and the fact that $W^*$ is homotopy equivalent to $X$, the proof of \cref{alexlines} is complete.
\end{proof}

\begin{rem}[Twisted Alexander polynomials of the boundary manifold]
Let $\cA=\{l_1,\ldots,l_m\}\subset \C^2$ be an essential line arrangement, and let $l_0=\C\P^2\backslash\C^2$ be the line at infinity. We consider the projective line arrangement $\cA'=\cA\cup\{l_0\}\subset \C\P^2$. The boundary manifold $B$ of the affine arrangement $\cA$ is the boundary of the manifold obtained by gluing balls around the singular points of the arrangement $\cA'$ and tubes around the smooth part of the lines, similar to what we did in the construction of $W^*$.

We have that the map induced by inclusion $\pi_0(B)\longrightarrow\pi_0(M)$ is an isomorphism, since both spaces are connected, and $\pi_1(B)\longrightarrow \pi_1(M)$ is an epimorphism, by a Lefschetz type argument. Thus, $M$ is obtained from $B$ by adjoining cells of dimension $\geq 2$, which as we have seen in the proof of \cref{dom} is enough to show that
$$
\Delta_0^{\vepsi,\rho}(M)=\Delta_0^{\vepsi,\rho}(B)
$$
and that
$$
\Delta_1^{\vepsi,\rho}(M)\text{ divides }\Delta_1^{\vepsi,\rho}(B),
$$
provided that $\Ht{1}{B}$ is a torsion $\Ft$-module. Moreover, following the proof of \cref{alexlines}, we conclude that $\Ht{1}{B}$ is indeed torsion (because $C_*^{\vepsi,\rho}(B,\F(t))$ is acyclic) and
$$
\frac{\Delta_1^{\vepsi,\rho}(B)}{\Delta_0^{\vepsi,\rho}(B)}=\left(\prod_{k=1}^s \deti(\beta_k)^{d_k-2}\right)\cdot\left(\prod_{i=0}^m\deti(a_i)^{\widetilde{s_i}-2}\right)
$$ 
where $s$ is the number of singular points of the \textbf{projective} arrangement, the $\beta_k$'s are certain distinguished loops near each of the singular points as in \cref{defbeta}, $\widetilde{s_i}$ is the number of singular points of the projective arrangement on the line $l_i$, and $a_i$ is a positively oriented meridian around the line $l_i$. This agrees with the result obtained in a different way by Cohen and Suciu in \cite[Theorem 5.2]{suciu} for multivariable twisted Alexander polynomials, as well as with the dibisibility result \cite[Theorem 5.6]{cogo} in the case where $\rho$ is a unitary representation. 

Note that, if $s_i$ is the number of singular points of the \textbf{affine} arrangement on the line $l_i$, for $i=1,\ldots, m$, then $\widetilde{s_i}=s_i+1$, so we can see that
$$
\Delta_1^{\vepsi,\rho}(W^*)\text{ divides }\Delta_1^{\vepsi,\rho}(B),
$$
and we conclude that the punctured stratified tubular neighborhood $W^*$ constitutes a better bound than the boundary manifold $B$ for the roots of the first twisted Alexander polynomial of $M$.
\label{remboundary}
\end{rem}

\begin{rem}
For \cref{zero} we do not need that $\vepsi$ be a positive epimorphism, just that it is a non-trivial map. For \cref{alexlines} to hold, we just need that $\vepsi$ takes non-zero values on the distinguished loops that appear in the formula of the twisted Alexander polynomial of $W^*$, as one can see in the proof.
\label{remepsi}
\end{rem}

In some cases, we can refine the result given by \cref{alexlines} as follows:

\linestube

\begin{proof}
Let $1\leq i\leq l$. Let $\B_{k,i}^4$ be a small ball around the singular point $P_k^i$, with boundary $S_{k,i}^3$, and let $L_{k_i}\subset S_{k,i}^3$ be the link of the singularity of $H$ at the point $P_k^i$. We will follow the proof of \cref{alexlines} and use the notation introduced there, but this time we define $X_i$ (instead of $X$) as the result of gluing $F_i\times S^1$ and $\bigsqcup\limits_{k=1}^{s_i}S_{k,i}^3\backslash L_{k_i}$ along the correspondig tori.

$X_i$ is connected, so the map induced by inclusion $\pi_0(X_i)\longrightarrow \pi_0(M)$ is an isomorphism. Moreover, since no other line in $\cA$ is parallel to $H_i$, we can see by a Lefschetz type argument that the map $\pi_1(X_i)\longrightarrow \pi_1(M)$ induced by inclusion is an epimorphism. Thus, $M$ is obtained from $X_i$ by adjoining cells of dimension $\geq 2$, which as we have seen in the proof of \cref{dom}, is enough to show that
$$
\Ht{1}{X_i}\longrightarrow \Ht{1}{M}
$$
is an epimorphism. Moreover, following the proof of \cref{alexlines}, we can show that all of the complexes involved in the Mayer-Vietoris short exact sequence of complexes with coefficients in $\F(t)$ except for $C_*^{\vepsi,\rho}(X_i,\F(t))$ are acyclic, so the long exact sequence in homology will tell us that $C_*^{\vepsi,\rho}(X_i,\F(t))$ is acyclic as well. In particular, $\Ht{1}{X_i}$ is torsion, and $\Delta_1^{\vepsi,\rho}(M)$ divides $\Delta_1^{\vepsi,\rho}(X_i)$ for all $i=1,\ldots, l$.

Following the proof of \cref{alexlines}, we get that
$$
\Delta_1^{\vepsi,\rho}(X_i)=\left(\prod_{k=1}^{s_i} \deti(\beta_{k,i})^{d_k^i-2}\right)\cdot\deti(a_i)^{s_i-1} \cdot \Delta_0^{\vepsi,\rho}(M)
$$
for all $i=1,\ldots, l$, and the result follows immediately by \cref{zero}.
%This result is going to be a direct application of the Fox Calculus method described in \cite[p. 6]{cogo}, as well as the presentation of the fundamental group of a plane curve complement given in \cite[Section 2]{lib2}.
\end{proof}

\subsection{Higher-dimensional case}
\label{ssgeneral}
\label{ssgeneralcase}
%We will follow the notation of \cite[Section 3]{libgober}.
Let $\mathcal{A}=\{H_1,\ldots,H_m\}$ be an essential hyperplane arrangement in $\C^n$. We consider the natural stratification of $H=\bigcup\limits_{i=1}^m H_i$, the one in which two points $P$ and $P'$ in $H$ lie in the same stratum if the collections of hyperplanes in the arrangement containing $P$ and $P'$ coincide. Let $\Sigma_1^k,\ldots,\Sigma_{s_k}^k$ be the collection of connected strata of (complex) dimension $k$. For each stratum, we define the multiplicity $m(\Sigma_i^k)$ as the number of hyperplanes in $\mathcal{A}$ containing a point from this stratum.

%Let $W$ be the closed tubular neighborhood of $H$ with boundary $V$ (as in Theorem \ref{tocho}), and let $U:=W\backslash (W\cap H)$. According to \cite[Theorem 2.3]{morse}, the inclusion $V\hookrightarrow U$ is a homotopy equivalence. Using the same argument following the flow lines of a normalized gradient field used to prove this last fact in \cite{morse}, one can also see that the inclusion $\Int(U)\hookrightarrow U$ is a homotopy equivalence.\\

\begin{rem}
By \cref{corroots}, the zeros of the $i$-th twisted Alexander polynomials of our arrangement complement $M$  are among the zeros of the $i$-th twisted Alexander polynomial of a punctured stratified tubular neighborhood $W^*$ of the arrangement, for $i=0,\ldots,n-1$. This observation prompts us to study what the zeros of the twisted Alexander polynomials of $W^*$ could be.
\label{rem}
\end{rem}

Let $$\{\mathcal{S}_l^k \mid k=0,\ldots,n-1; l=1,\ldots,s_k\}$$ be a collection of open subsets of $W^*$, each of which fibers over the corresponding stratum $\Sigma_l^k$, and chosen so that their union is $W^*$. These open subsets can be taken to be the tubular neighborhoods of open subsets of the strata that appeared in the construction of $W$ (before \cref{tube}) minus $H$. The fiber of
$
\mathcal{S}_l^k\longrightarrow \Sigma_l^k
$
is a central hyperplane arrangement complement consisting on $m(\Sigma_l^k)$ hyperplanes in $\C^{n-k}$. As it is pointed out in \cite[p. 5]{libgober}, if $k_1\geq k_2$, then $\mathcal{S}_{l_1}^{k_1}\cap\mathcal{S}_{l_2}^{k_2}$ is not empty if and only if the stratum $\Sigma_{l_2}^{k_2}$ is in the closure of the stratum $\Sigma_{l_1}^{k_1}$, and, in this intersection, the fibration that we consider is the one from $\mathcal{S}_{l_1}^{k_1}\rightarrow \Sigma_{l_1}^{k_1}$ restricted to it.

Let $\W:=\V^*=\Hom_\F(\V,\F)$ be the dual vector space of $\V$, and let $$\rho^*:\pi_1(M)\longrightarrow \GL(\W)$$
be the dual representation of $\rho:\pi_1(M)\longrightarrow \GL(\V)$, given by
$$
(w\cdot \alpha)(v)=w(v\cdot \alpha^{-1})
$$
for every $w\in \W$, $\alpha\in\pi_1(M)$ and $v\in\V$.

We consider the involution given by
$$
\f{\overline{\cdot}}{\Ft}{\Ft}{t}{\overline t:=t^{-1}}
$$
and we define the \textbf{conjugate $\Ft$-module structure} of an $\Ft$-module as the one obtained by composing the $\Ft$-module structure with the involution $\overline{\cdot}$. Then, as justified in \cite[p. 6 and p. 17]{maxtommy}, we have that
$$
\overline{H}^i(W^*,\cL_{\vepsi,\rho^*})\cong H^i\left(\Hom_{\Ft}\left(C_*^{\vepsi,\rho}(W^*,\Ft),\Ft\right)\right)
$$
for all $i$, where $\overline{H}^i(W^*,\cL_{\vepsi,\rho^*})$ means the $\Ft$-module $H^i(W^*,\cL_{\vepsi,\rho^*})$ with the conjugate module structure, and $\cL_{\vepsi,\rho^*}$ is the local system of $\Ft$-modules induced by the tensor representation $\vepsi\otimes\rho^*$. Therefore, the Universal Coefficient Theorem (UCT, for short) applied to the principal ideal domain $\Ft$ yields
$$
\overline{H}^i(W^*,\cL_{\vepsi,\rho^*})\cong\Hom_{\Ft}(\Ht{i}{W^*},\Ft)\oplus \Ext_{\Ft}(\Ht{i-1}{W^*},\Ft).
$$
Now, applying \cref{corroots} we get that $\Ht{i}{W^*}$ is a torsion $\Ft$-module for all $i\leq n-1$, so by the UCT, we get that
$$
\overline{H}^i(W^*,\cL_{\vepsi,\rho^*})\cong \Ext_{\Ft}(\Ht{i-1}{W^*},\Ft)\cong \Ht{i-1}{W^*}
$$
for all $i\leq n-1$.

\begin{rem}
In fact, we will see later on that $\Ht{i}{W^*}$ is a torsion $\Ft$-module for all $i$, so
$$
\overline{H}^i(W^*,\cL_{\vepsi,\rho^*})\cong \Ht{i-1}{W^*}
$$
for all $i$.
\label{remalli}
\end{rem}

Hence, the order of $H^i(W^*,\cL_{\vepsi,\rho^*})$ is $\overline{\Delta_{i-1}^{\vepsi,\rho}(W^*)(t)}=\Delta_{i-1}^{\vepsi,\rho}(W^*)(t^{-1})$ for all $i$. As indicated in \cref{rem}, we are interested in studying the zeros of $\Delta_{i}^{\vepsi,\rho}(W^*)$ for all $i$, or equivalently, the inverses of the zeros of the order of $H^{i+1}(W^*,\cL_{\vepsi,\rho^*})$ for all $i$.

Let us consider the Mayer-Vietoris spectral sequence (\cite[II.5.4]{godement}) of the sheaf $\mathcal{L}_{\vepsi,\rho^*}$ associated to the open covering
$$
\{\mathcal{S}_l^k\mid k=0,\ldots,n-1; l=1,\ldots,s_k\}.
$$
The first page of this spectral sequence is
\begin{equation}
E_1^{p,q}=\bigoplus H^q(\cS_{l_1}^{k_1}\cap\ldots\cap \cS_{l_{p+1}}^{k_{p+1}},\cL_{\vepsi,\rho^*}),
\label{ss1}
\end{equation}
where the direct sum is taken over all the possible non-empty intersections of $p+1$ open sets of our covering. This spectral sequence converges to $H^{p+q}(W^*,\cL_{\vepsi,\rho^*})$.

Let us study the elements in the first page of our Mayer-Vietoris spectral sequence, namely, the elements of the form
$$
H^q(\cS_{l_1}^{k_1}\cap\ldots\cap \cS_{l_{p+1}}^{k_{p+1}},\cL_{\vepsi,\rho^*}).
$$
Without loss of generality, we can assume that $k_1\geq\ldots\geq k_{p+1}$, and that $\Sigma_{l_j}^{k_j}\subset\overline{\Sigma_{l_1}^{k_1}}$ for every $j=2,\ldots,p+1$. Let $f:\cS_{l_1}^{k_1}\rightarrow \Sigma_{l_1}^{k_1}$ be the fibration. We consider a good cover $\cU$ of $f(\cS_{l_1}^{k_1}\cap\ldots\cap \cS_{l_{p+1}}^{k_{p+1}})$, which is an open set of the manifold $\Sigma_{l_1}^{k_1}$, and thus it is a manifold. By good cover we mean an open cover where all open sets and all finite intersections of those open sets are contractible. Let
$$
\cV:=\{f^{-1}(A) \mid A\in \cU\}
$$
We consider the Mayer Vietoris spectral sequence of the sheaf $\mathcal{L}_{\vepsi,\rho^*}$ associated to the open covering $\cV$. The first page of this spectral sequence is
\begin{equation}
E_1^{p,q}=\bigoplus H^q(B_{r_1}\cap\ldots\cap B_{r_{p+1}},\cL_{\vepsi,\rho^*})
\label{ss2}
\end{equation}
where the direct sum is taken over all the possible non-empty intersections $B_{r_1}\cap\ldots\cap B_{r_{p+1}}$ of $p+1$ open sets of our covering $\cV$. This spectral sequence converges to $H^{p+q}(\cS_{l_1}^{k_1}\cap\ldots\cap \cS_{l_{p+1}}^{k_{p+1}},\cL_{\vepsi,\rho^*})$.

Note that, since non-empty finite intersections of open sets in $\cU$ are contractible, and the map $f$ restricted to
$
\cS_{l_1}^{k_1}\cap\ldots\cap \cS_{l_{p+1}}^{k_{p+1}}
$
is a locally trivial fibration with fiber $F_{l_1,k_1}$, which is the complement of a central hyperplane arrangement consisting on $m(\Sigma_{l_1}^{k_1})$ hyperplanes in $\C^{n-k_1}$, then
$
B_{r_1}\cap\ldots\cap B_{r_{p+1}}
$
is homeomorphic to the product of $F_{l_1,k_1}$ and a contractible open set, and thus it is homotopy equivalent to $F_{l_1,k_1}$. Thus,
$$
H^q(B_{r_1}\cap\ldots\cap B_{r_{p+1}},\cL_{\vepsi,\rho^*})\cong H^q(F_{l_1,k_1},\cL_{\vepsi,\rho^*}).
$$

Note that central hyperplane arrangements are homotopy equivalent to the complement in $S^{2n-1}$ of their links at infinity, so by \cite[proof of Theorem 4.1]{maxtommy}, $\Ht{q}{F_{l_1,k_1}}$ (and consequently $H^q(F_{l_1,k_1},\cL_{\vepsi,\rho^*})$) are torsion for all $q$. This implies that all of the elements in spectral sequence (\ref{ss2}) are torsion modules. In fact, by \cite[Theorem 4.11]{maxtommy}, the zeros of the order of $H^q(F_{l_1,k_1},\cL_{\vepsi,\rho^*})$ are among those of the order of the cokernel of the endomorphism
$$t^{-\vepsi(\gamma_\infty(F_{l_1,k_1}))}\rho^*(\gamma_\infty(F_{l_1,k_1}))^{-1}-\Id\in\End(\Ft\otimes_\F \V)$$
where $\gamma_\infty(F_{l_1,k_1})$ is a loop around the hyperplane at infinity in $(\C\P)^{n-k_1}$. Note that $\vepsi$ restricted to $\pi_1(F_{l_1,k_1})$ is not necessarily an epimorphism, but the proof of \cite[Theorem 4.11]{maxtommy} does not require it.

Let us try to describe these ``loops at infinity'' that we are using in more detail. Let $H_{t_1},\ldots,H_{t_{m(\Sigma_{l_1}^{k_1})}}$ be the hyperplanes going through the stratum $\Sigma_{l_1}^{k_1}$ associated to the fiber $F_{l_1,k_1}$. We have that $\gamma_\infty(F_{l_1,k_1})$ has an expression of the form
$$
(\gamma_{t_1}\cdot\ldots\cdot\gamma_{t_{m(\Sigma_{l_1}^{k_1})}})^{-1}
$$
where $\gamma_{t_1},\ldots,\gamma_{t_{m(\Sigma_{l_1}^{k_1})}}$ are an appropriate choice of meridians around each component of the central hyperplane arrangement given by $F_{l_1,k_1}$ in the appropriate order.

%By spectral sequence (\ref{ss2}), we have that the $H^{p+q}(\cS_{l_1}^{k_1}\cap\ldots\cap \cS_{l_{p+1}}^{k_{p+1}},\cL_{\vepsi,\rho}^{\vee})$ are torsion and the zeroes of their order lie among those of the order of the cokernels of $t^{\vepsi(\gamma_\infty(F))}\rho(\gamma_\infty(F))-Id$, where $F$ ranges among all the possible fibers corresponding to the fibrations over strata. There are a finite number of strata, so there are a finite number of such fibers $F$.

Note that $\vepsi(\gamma_\infty(F_{l_1,k_1}))<0$, so the order of the cokernel of $$t^{-\vepsi(\gamma_\infty(F_{l_1,k_1}))}\rho^*(\gamma_\infty(F_{l_1,k_1}))^{-1}-\id$$
is exactly
$$
\det(t^{-\vepsi(\gamma_\infty(F_{l_1,k_1}))}\rho^*(\gamma_\infty(F_{l_1,k_1}))^{-1}-\id)
$$

By the discussion above, the zeros of the order of $H^q(F_{l_1,k_1},\cL_{\vepsi,\rho^*})$ are among those of $$
\det(t^{-\vepsi(\gamma_\infty(F_{l_1,k_1}))}\rho^*(\gamma_\infty(F_{l_1,k_1}))^{-1}-\id).
$$
Using the spectral sequence \eqref{ss2}, we get that the zeros of the order of $$H^{q}(\cS_{l_1}^{k_1}\cap\ldots\cap \cS_{l_{p+1}}^{k_{p+1}},\cL_{\vepsi,\rho^*})$$
are among the zeros of $\det(t^{-\vepsi(\gamma_\infty(F_{l_1,k_1}))}\rho^*(\gamma_\infty(F_{l_1,k_1}))^{-1}-\id)$.

Now, by using spectral sequence (\ref{ss1}), we see that $H^i (W^*,\cL_{\vepsi,\rho^*})$ is torsion for all $i$. By the Universal Coefficient Theorem, this means that $\Ht{i}{W^*}$ is also torsion for all $i$, as anticipated in \cref{remalli}. Moreover, 
the zeros of the order of $H^{q}(W^*,\cL_{\vepsi,\rho^*})$ are among the zeros of
$$
\prod_{k=0}^{n-1}\prod_{l=1}^{s_k}\det(t^{-\vepsi(\gamma_\infty(F_{l,k}))}\rho^*(\gamma_\infty(F_{l,k}))^{-1}-\id)
$$
for all $q$.

Hence, by \cref{rem} and \cref{remalli}, and the fact that $\rho^*(\alpha)^{-1}=\rho(\alpha)^{T}$ for every $\alpha\in\pi_1(M)$ (seen as matrices in $\GL_n(\F)$), we get that the zeros of $\Delta_i^{\vepsi,\rho}(M)$ are among the zeros of
$$
\prod_{k=0}^{n-1}\prod_{l=1}^{s_k}\det(t^{\vepsi(\gamma_\infty(F_{l,k}))}\rho(\gamma_\infty(F_{l,k}))-\id).
$$

Thus, we have arrived to the following result.

\rootsgeneral

We can see that this result generalizes the one obtained in the line arrangement case. The meridians around the hyperplane at infinity in the line arrangement case are $\beta_k^{-1}$ ($k=1,\ldots, s$), which correspond to the 0-dimensional strata (the singular points); and $a_i^{-1}$ ($i=1,\ldots, m$), which correspond to the 1-dimensional strata.

\section{Applications}
\label{saplications}

\subsection{Topology of a hyperplane arrangement complement via twisted Alexander polynomials.}
\label{sstopology}

In the following example, we discuss how twisted Alexander polynomials can give us information about the topology of the complement of a line arrangement. In particular, they can be used to distinguish the homeomorphism type of certain line arrangement complements that are homotopy equivalent.

\begin{exm}
Let us consider a pair of line arrangements (the Falk arrangements) $\cA_1$ and $\cA_2$ shown in \cref{falk}, which are given by the zeros of $p_1(x,y)$ and $p_2(x,y)$ respectively, where
$$
\begin{array}{c}
p_1(x,y)=(x+1)(x-1)(x+y)y(x-y)\\
p_2(x,y)=(x+1)(x-1)(y+1)(y-1)(x-y-1).\\
\end{array}
$$

\begin{figure}[h]
\centering
\begin{minipage}{0.46\textwidth}
\centering
\begin{tikzpicture}
\draw (-1,1.5) -- (-1,-1.5)node[pos=1,anchor=north]{$l_1$};
\draw (1,1.5) -- (1,-1.5)node[pos=1,anchor=north]{$l_2$};
\draw (-1.5,1.5) -- (1.5,-1.5)node[pos=1,anchor=north west]{$l_3$};
\draw (-1.5,0) -- (1.5,0)node[pos=1,anchor=north west]{$l_4$};
\draw (-1.5,-1.5) -- (1.5,1.5)node[pos=1,anchor=north west]{$l_5$};
\end{tikzpicture}
\end{minipage}
\begin{minipage}{0.46\textwidth}
\centering
\begin{tikzpicture}[scale=0.6]
\draw (-1,2.5) -- (-1,-2.5)node[pos=1,anchor=north]{$l_1$};
\draw (1,2.5) -- (1,-2.5)node[pos=1,anchor=north]{$l_2$};
\draw (-2.5,-1) -- (2.5,-1)node[pos=1,anchor=north west]{$l_3$};
\draw (-2.5,1) -- (2.5,1)node[pos=1,anchor=north west]{$l_4$};
\draw (-1.5,-2.5) -- (2.5,1.5)node[pos=1,anchor=south]{$l_5$};
\end{tikzpicture}
\end{minipage}
\caption{The real part of the Falk Arrangements $\cA_1$ and $\cA_2$.}
\label{falk}
\end{figure}

In \cite{falk}, Falk showed that the complements of these two arrangements are homotopy equivalent even though they are combinatorially quite different. These two complements are not homeomorphic, as shown by Jiang and Yau in \cite{yau}. In \cite{suciu}, Cohen and Suciu reproved that the complements are not homeomorphic by showing that the boundary manifolds of $\cA_1$ and $\cA_2$ are not homotopy equivalent, which they did by showing that their corresponding multivariable Alexander polynomials had a different number of distinct factors.

We will show that the boundary manifolds of $\cA_1$ and $\cA_2$ are not homotopically equivalent by showing that certain (one-variable) twisted Alexander polynomials of their boundary manifolds have a different number of distinct roots with multiplicity, thus reproving the result by Jiang and Yau by using simpler invariants.

\begin{proof}
Let $B_j$ be the boundary manifold of $\cA_j$, for $j=1,2$. Let $M_j$ be the complement in $\C^2$ of the arrangement $\cA_j$, for $j=1,2$. We will argue by contradiction. Let us assume that there exists a homotopy equivalence $$h:B_1\longrightarrow B_2$$
We denote by $h_*$ the map that $h$ induces on fundamental groups. Let $i_2:B_2\hookrightarrow M_2$ be the inclusion and $(i_2)_*$ the map it induces on fundamental groups. Let $$\vepsi:\pi_1(M_2)\longrightarrow \Z$$
be an epimorphism, and let
$$
\rho:\pi_1(M_2)\longrightarrow \C^*
$$
be a one dimensional representation. Restricting ourselves to one dimensional representations makes computing twisted Alexander polynomials so much easier, since they factor through the abelianization of $\pi_1(M_2)$ and we do not have to care about the conjugation of meridians due to the braiding in the fundamental group.

We abuse notation and also call $\vepsi$ and $\rho$ the maps induced by pulling back  $\vepsi$ and $\rho$ by $(i_2)_*$ and $(i_2)_*\circ h_*$ on $\pi_1(B_2)$ and $\pi_1(B_1)$ respectively. In this setting, since $h$ is a homotopy equivalence, we have that
$$
\frac{\Delta_1^{\vepsi,\rho}(B_1)}{\Delta_0^{\vepsi,\rho}(B_1)}=\frac{\Delta_1^{\vepsi,\rho}(B_2)}{\Delta_0^{\vepsi,\rho}(B_2)}
$$
up to multiplication by a unit of $\C[t^{\pm 1}]$. Thus, the set of non-zero roots with multiplicity corresponding to both sides should be the same. We will show that for some choice of $\vepsi$ and $\rho$, they are not, which will conclude our proof.

Let $a_i^j$ be a meridian around the line given by the $i$-th factor of $p_j(x,y)$, and let $a_0^j=\left(\prod_{i=1}^5 a_i^j\right)^{-1}$, for $j=1,2$ and $i=1,\ldots, 5$. The loop $a_0^j$ is not necessarily a meridian around the line at infinity (due to the order chosen in the multiplication), but will have the same image by $\vepsi$ and $\rho$ than said meridian. We denote by $a_{i_1 i_2 i_3}^j=a_{i_1}^j\cdot a_{i_2}^j \cdot a_{i_3}^j$. Note that, if the lines $l_{i_1}$, $l_{i_2}$ and $l_{i_3}$ intersect in a triple point, then $a_{i_1 i_2 i_3}^j$ will have the same image by $\vepsi$ and $\rho$ than the corresponding $\beta_k$.

We choose $\vepsi:\pi_1(M_2)\longrightarrow \Z$ to be an epimorphism such that all of the loops involved in the formulas for $\frac{\Delta_1^{\vepsi,\rho}(B_2)}{\Delta_0^{\vepsi,\rho}(B_2)}$ and $\frac{\Delta_1^{\vepsi,\rho}(B_1)}{\Delta_0^{\vepsi,\rho}(B_1)}$ given by \cref{remboundary} have a non-zero image by $\vepsi$ (recall \cref{remepsi}). This choice of $\vepsi$ depends on $h_*$, and generically, this condition on $\vepsi$ is satisfied. In that case, we have that

$$
\frac{\Delta_1^{\vepsi,\rho}(B_2)}{\Delta_0^{\vepsi,\rho}(B_2)}=
\left(\prod_{i=1}^4 \rho(a_i^2)t^{\vepsi(a_i^2)}-1\right)^2\cdot(\rho(a_5^2)t^{\vepsi(a_5^2)}-1)^3\cdot (\rho(a_0^2)t^{\vepsi(a_0^2)}-1)\cdot(\rho(a_{034}^2)t^{\vepsi(a_{034}^2)}-1)\cdot(\rho(a_{012}^2)t^{\vepsi(a_{012}^2)}-1)
$$
so, up to multiplication by a unit of $\C[t^{\pm 1}]$,
\begin{equation}
\frac{\Delta_1^{\vepsi,\rho}(B_2)}{\Delta_0^{\vepsi,\rho}(B_2)}=\left(\prod_{i=1}^4 \rho(a_i^2)t^{\vepsi(a_i^2)}-1\right)^2\cdot(\rho(a_5^2)t^{\vepsi(a_5^2)}-1)^3 \cdot(\rho(a_0^2)t^{\vepsi(a_0^2)}-1)\cdot(\rho(a_{125}^2)t^{\vepsi(a_{125}^2)}-1)\cdot(\rho(a_{345}^2)t^{\vepsi(a_{345}^2)}-1).
\label{eqnB2}
\end{equation}
Also, we have that
$$
\frac{\Delta_1^{\vepsi,\rho}(B_1)}{\Delta_0^{\vepsi,\rho}(B_1)}=\left(\prod_{i=0}^5 \left(\rho(a_i^1)t^{\vepsi(a_i^1)}-1\right)\right)^2(\rho(a_{012}^1)t^{\vepsi(a_{012}^1)}-1)(\rho(a_{345}^1)t^{\vepsi(a_{345}^1)}-1).
$$
Note that, up to multiplication by a unit of $\C[t^{\pm 1}]$, the last two factors are the same, so
\begin{equation}
\frac{\Delta_1^{\vepsi,\rho}(B_1)}{\Delta_0^{\vepsi,\rho}(B_1)}=\left(\prod_{i=0}^5 \left(\rho(a_i^1)t^{\vepsi(a_i^1)}-1\right)\right)^2(\rho(a_{345}^1)t^{\vepsi(a_{345}^1)}-1)^2.
\label{eqnB1}
\end{equation}

Now that we have fixed $\vepsi$, we can choose $\rho$ so that any root of  $(\rho(a)t^{\vepsi(a)}-1)$ is different than any root of $(\rho(b)t^{\vepsi(b)}-1)$ for different loops $a$ and $b$ involved in the formula (\ref{eqnB2}). That way, if we pick a given root of the term $(\rho(a_5^2)t^{\vepsi(a_5^2)}-1)^3$, we know it only appears $3$ times as a root of $\frac{\Delta_1^{\vepsi,\rho}(B_2)}{\Delta_0^{\vepsi,\rho}(B_2)}$. On the other hand, no non-zero root of $\frac{\Delta_1^{\vepsi,\rho}(B_1)}{\Delta_0^{\vepsi,\rho}(B_1)}$ can have odd multiplicity, so we have reached a contradiction.
\end{proof}
\label{exmfalk}
\end{exm}

\subsection{Twisted jump loci vs. twisted Alexander polynomials}

Let $\cA=\{H_1,\ldots, H_m\}$ be an essential hyperplane arrangement in $\C^n$, let $H=\bigcup\limits_{i}^m H_i$, and let $M=\C^n\backslash H$ be the arrangement complement in $\C^n$. Let $\V$ be an $n$-dimensional vector space over $\C$, and let
$$
\rho:\pi_1(M)\longrightarrow \GL(\V)
$$
be a representation. We denote by $V_\rho$ the corresponding $\V$-local system on $M$.

\begin{defn}
The \textbf{rank 1 homology jump loci of $M$ twisted by $\rho$} are defined to be
$$
\cV_i^k(M,\rho)=\left\{\eta\in\Hom(\pi_1(M),\C^*)\mid \dim_\C H_i(M,\cL_\eta\otimes V_\rho)\geq k\right\}
$$
for all $i,k\geq 0$, where $\cL_\eta$ is the rank $1$ local system on $M$ defined by $\eta$.
\end{defn}

There exists a natural isomorphism
$$
(\C^*)^m \xrightarrow{\cong}\Hom(\pi_1(M),\C^*)
$$
that takes any tuple $(z_1,\ldots,z_m)\in(\C^*)^m$ to the unique morphism that sends the positively oriented meridians around the line $H_i$ to $z_i$ for all $i=1,\ldots, m$. In this way, we can see the homology jump loci $\cV_i^k(M,\rho)$ inside of $(\C^*)^m$.

\begin{rem}
Suppose that $\rho$ is a one dimensional representation, which corresponds to $(\rho_1,\ldots,\rho_m)$ in $(\C^*)^m$ via the natural isomorphism described above. Let $$\cV_i^k(M):=\left\{\eta\in\Hom(\pi_1(M),\C^*)\mid \dim_\C H_i(M,\cL_\eta)\geq k\right\}$$ be the untwisted homology jump loci. Then, we have the following correspondence between subsets of $(\C^*)^m$:
$$
(x_1,\ldots,x_m)\in \cV_i^k(M,\rho) \Leftrightarrow (x_1\rho_1,\ldots,x_m\rho_m)\in \cV_i^k(M).
$$
\end{rem}

Let $\vepsi:\pi_1(M)\longrightarrow \Z$ be an epimorphism. It induces the following map
$$
\f{\vepsi^*}{\C^*}{\Hom(\pi_1(M),\C^*)}{a}{h_a\circ\vepsi}
$$
where $h_a:\Z\longrightarrow \C^*$ is the only group homomorphism taking $1$ to $a$. Since $\vepsi$ is an epimorphism, we have that the image of $\vepsi^*$ is naturally isomorphic to $\C^*$. With this notation, we have the following result that relates the zeros of twisted Alexander polynomials of $M$ and the twisted rank 1 homology jump loci.

\begin{prop}
Let $\F=\C$. Then,
$$\{t\in\C^*\mid\Delta_i^{\vepsi,\rho}(M)(t)\cdot\Delta_{i-1}^{\vepsi,\rho}(M)(t)=0\}=\cV_i^1(M,\rho)\cap \Ima(\vepsi^*)$$
for $0\leq i\leq n-1$, and
$$
\{t\in\C^*\mid\Delta_{n-1}^{\vepsi,\rho}(M)(t)=0\}=\cV_{n}^{(\dim_\C\V\cdot|\chi(M)|+1)}(M,\rho)\cap \Ima(\vepsi^*)
$$
where $\cV_i^k(M,\rho)\cap \Ima(\vepsi^*)$ is seen as a subset of $\C^*$.
\end{prop}

\begin{proof}
We will follow the notation in \cite[Theorem 4.5]{dn}, where the non-twisted case is discussed.

Let $a\in\C^*$. The homomorphism $\vepsi^*(a)$ defines a $1$-dimensional local system, which we will call $\cL_a$. We consider the following short exact sequence of vector spaces over $\C$:
$$
0\longrightarrow \C[t^{\pm 1}]\xrightarrow{t-a} \C[t^{\pm 1}] \xrightarrow{t=a} \C\longrightarrow 0
$$

Tensoring by $\V$, we obtain the following short exact sequence of vector spaces over $\C$:
\begin{equation}
0\longrightarrow \C[t^{\pm 1}]\otimes_{\C}\V\xrightarrow{f} \C[t^{\pm 1}]\otimes_{\C}\V \xrightarrow{g} \C\otimes_{\C}\V\longrightarrow 0
\label{eqnses}
\end{equation}
The vector space $\C[t^{\pm 1}]\otimes_\C\V$ can be given the structure of a right $\C[\pi_1(M)]$-module, as we described in \cref{defntwistmod}. Moreover, $\C\otimes_\C\V\cong \V$ can also be given the structure of a right $\C[\pi_1(M)]$-module, with the right action given by
$$
v\cdot \alpha=a^{\vepsi(\alpha)}v \cdot \rho(\alpha)
$$
for every $v\in\V$ and $\alpha\in \pi_1(X)$, where $v$ is regarded as a row vector and $\rho(\alpha)$ as a square matrix.

We can check that both $f$ and $g$ respect the right $\C[\pi_1(M)]$-module structure, so the short exact sequence (\ref{eqnses}) is also a short exact sequence of right $\C[\pi_1(M)]$-modules.

Let $\widetilde{M}$ be the universal cover of $M$. We have that $C_i(\widetilde{M},\C)$ is a free left $\C[\pi_1(M)]$-module for all $i\in\Z$, as explained in \cref{defntwistmod}. In particular, it is flat, so we can tensor (\ref{eqnses}) by $C_i(\widetilde{M},\C)$ to get
$$
0\longrightarrow C_i^{\vepsi,\rho}(M,\C[t^{\pm 1}])\longrightarrow C_i^{\vepsi,\rho}(M,\C[t^{\pm 1}])\longrightarrow \V\otimes_{\C[\pi_1(M)]} C_i(\widetilde{M},\C) \longrightarrow 0
$$

These short exact sequences for $i\in\Z$ extend to a short exact sequence of complexes (i.e. they are compatible with the differentials), so we get the corresponding long exact sequence in homology, namely
\begin{equation}
\ldots\longrightarrow\Htc{i}{M}\xrightarrow{t-a}\Htc{i}{M}\longrightarrow H_i(M,\cL_a\otimes V_\rho)\longrightarrow \Htc{i-1}{M}\longrightarrow \ldots
\label{eqnles}
\end{equation}

By \cref{torsion} and the fact that $\C[t^{\pm 1}]$ is a principal ideal domain, we get that, for $0\leq i\leq n-1$, the twisted Alexander modules have a primary decomposition of the form
$$
\Htc{i}{M}\cong \C[t^{\pm 1}]/\left((t-b_1)^{r_1}\right)\oplus\ldots\oplus \C[t^{\pm 1}]/\left((t-b_{l_i})^{r_{l_i}}\right).
$$

Let $N(a,i)$ be the number of direct summands in the $(t-a)$-torsion part of $\Htc{i}{M}$. We have that
$$
N(a,i)=\dim_\C \ker\left(\Htc{i}{M}\xrightarrow{t-a}\Htc{i}{M}\right).
$$

Let us consider (\ref{eqnles}) as a long exact sequence of vector spaces. By a dimension counting argument, we deduce that
$$
\dim_\C H_i(M,\cL_a\otimes V_\rho)=N(a,i)+N(a,i-1)
$$
for $0\leq i\leq n-1$. Note that $a$ is a zero of $\Delta_l^{\vepsi,\rho}(M)$ if and only if $N(a,l)\geq 1$. Thus
$$
\{t\in\C^*\mid\Delta_i^{\vepsi,\rho}(M)(t)\cdot\Delta_{i-1}^{\vepsi,\rho}(M)(t)=0\}=\cV_i^1(M,\rho)\cap \Ima(\vepsi^*)
$$
for $0\leq i\leq n-1$.

Taking into account that $\Htc{n}{M}$ is a free $\C[t^{\pm 1}]$-module of dimension $\dim_\C\V\cdot|\chi(M)|$ (\cref{torsion}), by a dimension counting argument in the long exact sequence (\ref{eqnles}), we have that
$$
\dim_\C H_n(M,\cL_a\otimes V_\rho)=\dim_\C\V\cdot|\chi(M)|+N(a,n-1).
$$
Thus,
$$
\{t\in\C^*\mid\Delta_{n-1}^{\vepsi,\rho}(M)(t)=0\}=\cV_{n}^{(\dim_\C\V\cdot|\chi(M)|+1)}(M,\rho)\cap \Ima(\vepsi^*).
$$
\end{proof}

The result that we just proved, along with the main results of \cref{sroots}, can give us some information about the rank 1 twisted homology jump loci of $M$. More specifically, the following corollaries follow from \cref{zero}, \cref{alexlines}, \cref{linestube} and \cref{rootsgeneral} respectively.

\begin{cor}
Let $\F=\C$. Using the same notation and under the same assumptions of \cref{zero}, we have that the points of
$$
\cV_0^1(M,\rho)\cap \Ima(\vepsi^*)
$$
are in one-to-one correspondence with the common roots of all of the dimension $\dim_\C\V$ minors of the column matrix with entries $$t^{\vepsi(a_i)}\rho(a_i)-\id\in \cM_{(\dim_\C\V)\times (\dim_\C\V)}(\C[t^{\pm 1}])$$
for $i=1,\ldots, m$.
\end{cor}

\begin{cor}
Let $\F=\C$. Using the same notation and under the same assumptions of \cref{alexlines}, we have that both
$$
\cV_1^1(M,\rho)\cap \Ima(\vepsi^*)
$$
and
$$
\cV_2^{(\dim_\C\V\cdot|\chi(M)|+1)}(M,\rho)\cap \Ima(\vepsi^*)
$$
are contained in
$$
\left\{t\in\C^*\mid\left(\prod_{k=1}^s \deti(\beta_k)^{d_k-2}\right)\cdot\left(\prod_{i=1}^m\deti(a_i)^{s_i-1}\right)\cdot\Delta_0^{\vepsi,\rho}(M)(t)=0\right\}.
$$
\end{cor}

\begin{cor}
Let $\F=\C$. Using the same notation and under the same assumptions of \cref{linestube}, we have that both
$$
\cV_1^1(M,\rho)\cap \Ima(\vepsi^*)
$$
and
$$
\cV_2^{(\dim_\C\V\cdot|\chi(M)|+1)}(M,\rho)\cap \Ima(\vepsi^*)
$$
are contained in
$$
\bigcap_{i=1}^l\left\{t\in\C^*\mid\Delta_0^{\vepsi,\rho}(M)(t)\cdot\left(\prod_{k=1}^{s_i} \deti(\beta_{k,i})^{d_k^i-2}\right)\cdot\deti(a_i)^{s_i-1}=0\right\}.
$$
\end{cor}

\begin{cor}
Let $\F=\C$. Using the same notation and under the same assumptions of \cref{rootsgeneral}, we have that
$$
\cV_i^1(M,\rho)\cap \Ima(\vepsi^*)
$$
for $0\leq i\leq n-1$, and
$$
\cV_n^{(\dim_\C\V\cdot|\chi(M)|+1)}(M,\rho)\cap \Ima(\vepsi^*)
$$
are all contained in
$$
\left\{t\in\C^*\mid
\prod_{k=0}^{n-1}\prod_{l=1}^{s_k}\deti(\gamma_\infty(F_{l,k}))
=0\right\}.
$$
\end{cor}

%%%%%%%%%%%%%%%%%%%%%%%%%%%%%%%%%%%%%%%%%%%%%%%%%%%%%%%%%%%%%%%%%%%%%%%%%%%%%%%%%%
%%%%%%%%%%%%%%%%%%%%%%%%%%%%%%%%%%%%%%%%%%%%%%%%%%%%%%%%%%%%%%%%%%%%%%%%%%%%%%%%%%
%%%%%%%%%%%%%%%%%%%%%%%%%%ACKNOWLEDGEMENTS%%%%%%%%%%%%%%%%%%%%%%%%%%%%%%%%%%%%%%%%
%%%%%%%%%%%%%%%%%%%%%%%%%%%%%%%%%%%%%%%%%%%%%%%%%%%%%%%%%%%%%%%%%%%%%%%%%%%%%%%%%%
%%%%%%%%%%%%%%%%%%%%%%%%%%%%%%%%%%%%%%%%%%%%%%%%%%%%%%%%%%%%%%%%%%%%%%%%%%%%%%%%%%

\section*{Acknowledgements}
The author would like to thank Lauren\c{t}iu Maxim for all of his guidance and support during this project. She would also like to acknowledge Enrique Artal-Bartolo, Jos\' e Ignacio Cogolludo-Agust\' in, and Miguel \' Angel Marco-Buzun\' ariz for the interesting discussions we had about this topic.

%%%%%%%%%%%%%%%%%%%%%%%%%%%%%%%%%%%%%%%%%%%%%%%%%%%%%%%%%%%%%%%%%%%%%%%%%%%%%%%%%%
%%%%%%%%%%%%%%%%%%%%%%%%%%%%%%%%%%%%%%%%%%%%%%%%%%%%%%%%%%%%%%%%%%%%%%%%%%%%%%%%%%
%%%%%%%%%%%%%%%%%%%%%%%%%%BIBLIOGRAPHY%%%%%%%%%%%%%%%%%%%%%%%%%%%%%%%%%%%%%%%%%%%%
%%%%%%%%%%%%%%%%%%%%%%%%%%%%%%%%%%%%%%%%%%%%%%%%%%%%%%%%%%%%%%%%%%%%%%%%%%%%%%%%%%
%%%%%%%%%%%%%%%%%%%%%%%%%%%%%%%%%%%%%%%%%%%%%%%%%%%%%%%%%%%%%%%%%%%%%%%%%%%%%%%%%%

\Addresses


\begin{bibdiv}
\begin{biblist}
\bib{arv}{article}{
   author={Arvola, W. A.},
   title={The fundamental group of the complement of an arrangement of
   complex hyperplanes},
   journal={Topology},
   volume={31},
   date={1992},
   number={4},
   pages={757--765}
}

%\bib{cogozvk}{article}{
   %author={Cogolludo-Agust\'\i n, J. I.},
   %title={Braid monodromy of algebraic curves},
   %journal={Ann. Math. Blaise Pascal},
   %volume={18},
   %date={2011},
   %number={1},
   %pages={141--209}
%}

\bib{cogo}{article}{
   author={Cogolludo-Agust\'\i n, J. I.},
   author={Florens, V.},
   title={Twisted Alexander polynomials of plane algebraic curves},
   journal={J. Lond. Math. Soc. (2)},
   volume={76},
   date={2007},
   number={1},
   pages={105--121}
}

\bib{suciu}{article}{
   author={Cohen, D. C.},
   author={Suciu, A. I.},
   title={The boundary manifold of a complex line arrangement},
   book={
      series={Geom. Topol. Monogr.},
      volume={13},
      publisher={Geom. Topol. Publ., Coventry},
   },
   date={2008},
   pages={105--146}
}

\bib{dimca}{book}{
   author={Dimca, A.},
   title={Sheaves in topology},
   series={Universitext},
   publisher={Springer-Verlag, Berlin},
   date={2004}
}

\bib{dimca2}{book}{
   author={Dimca, A.},
   title={Singularities and topology of hypersurfaces},
   series={Universitext},
   publisher={Springer-Verlag, New York},
   date={1992}
}

\bib{dl}{article}{
   author={Dimca, A.},
   author={Libgober, A.},
   title={Regular functions transversal at infinity},
   journal={Tohoku Math. J. (2)},
   volume={58},
   date={2006},
   number={4},
   pages={549--564}
}

\bib{dn}{article}{
   author={Dimca, A.},
   author={N\'emethi, A.},
   title={Hypersurface complements, Alexander modules and monodromy},
   book={
      series={Contemp. Math.},
      volume={354},
      publisher={Amer. Math. Soc., Providence, RI},
   },
   date={2004},
   pages={19--43}
}

\bib{falk}{article}{
   author={Falk, M.},
   title={Homotopy types of line arrangements},
   journal={Invent. Math.},
   volume={111},
   date={1993},
   number={1},
   pages={139--150}
}

\bib{mam}{article}{
   author={Florens, V.},
   author={Guerville-Ball\'e, B.},
   author={Marco-Buzunariz, M. A.},
   title={On complex line arrangements and their boundary manifolds},
   journal={Math. Proc. Cambridge Philos. Soc.},
   volume={159},
   date={2015},
   number={2},
   pages={189--205}
}
\bib{godement}{book}{
   author={Godement, Roger},
   title={Topologie alg\'{e}brique et th\'{e}orie des faisceaux},
   series={Actualit'es Sci. Ind. No. 1252. Publ. Math. Univ. Strasbourg. No.
   13 },
   publisher={Hermann, Paris},
   date={1958},
}

\bib{hironaka}{article}{
   author={Hironaka, E.},
   title={Boundary manifolds of line arrangements},
   journal={Math. Ann.},
   volume={319},
   date={2001},
   number={1},
   pages={17--32}
}

\bib{yau}{article}{
   author={Jiang, T.},
   author={Yau, S. S.-T.},
   title={Intersection lattices and topological structures of complements of
   arrangements in ${\bf C}{\rm P}^2$},
   journal={Ann. Scuola Norm. Sup. Pisa Cl. Sci. (4)},
   volume={26},
   date={1998},
   number={2},
   pages={357--381}
}

\bib{kirk}{article}{
   author={Kirk, P.},
   author={Livingston, C.},
   title={Twisted Alexander invariants, Reidemeister torsion, and
   Casson-Gordon invariants},
   journal={Topology},
   volume={38},
   date={1999},
   number={3},
   pages={635--661}
}

\bib{morse}{article}{
   author={Kohno, T.},
   author={Pajitnov, A.},
   title={Circle-valued Morse theory for complex hyperplane arrangements},
   journal={Forum Math.},
   volume={27},
   date={2015},
   number={4},
   pages={2113--2128}
}

\bib{libgober}{article}{
   author={Libgober, A.},
   title={Eigenvalues for the monodromy of the Milnor fibers of
   arrangements},
   conference={
      title={Trends in singularities},
   },
   book={
      series={Trends Math.},
      publisher={Birkh\"auser, Basel},
   },
   date={2002},
   pages={141--150}
}

\bib{lib2}{article}{
   author={Libgober, A.},
   title={On the homotopy type of the complement to plane algebraic curves},
   journal={J. Reine Angew. Math.},
   volume={367},
   date={1986},
   pages={103--114}
}

\bib{lib3}{article}{
   author={Libgober, A.},
   title={The topology of complements to hypersurfaces and nonvanishing of a
   twisted de Rham cohomology},
   book={
      series={AMS/IP Stud. Adv. Math.},
      volume={5},
      publisher={Amer. Math. Soc., Providence, RI},
   },
   date={1997},
   pages={116--130}
}

\bib{maxtommy}{article}{
   author={Maxim, L.},
   author={Wong, K.},
   title={Twisted Alexander invariants of complex hypersurface complements},
   journal={Proc. Roy. Soc. Edinburgh Sect. A},
   date={2018},
   pages={1--25}
}

\bib{milmor}{book}{
   author={Milnor, J.},
   title={Morse theory},
   series={Based on lecture notes by M. Spivak and R. Wells. Annals of
   Mathematics Studies, No. 51},
   publisher={Princeton University Press, Princeton, N.J.},
   date={1963}
}


\bib{milnor}{article}{
   author={Milnor, J.},
   title={Whitehead torsion},
   journal={Bull. Amer. Math. Soc.},
   volume={72},
   date={1966},
   pages={358--426}
}

%\bib{shimada}{article}{
   %author={Shimada, I.},
   %title={On the Zariski-van Kampen theorem},
   %journal={Canad. J. Math.},
   %volume={55},
   %date={2003},
   %number={1},
   %pages={133--156}
%}

\end{biblist}
\end{bibdiv}
\end{document}